\documentclass[journal]{IEEEtran}

\usepackage{amsmath}
\usepackage{amssymb}
\usepackage{amsthm}
\usepackage[dvips]{graphicx}
\usepackage{color}
\usepackage{subfigure}
\usepackage{bm}
\usepackage{hyperref}

\newtheorem{RK}{Remark}

\newtheorem{define}{Definition}

\newtheorem{ppt}{Property}
\newtheorem*{ntn*}{Notation}
\newtheorem{ntn}{Notation}
\newtheorem{lem}{Lemma}
\newtheorem{thm}{Theorem}

\newtheorem{cor}{Corollary}

\newtheorem{fct}{Fact}

\newcommand{\nn}{\nonumber}
\newcommand{\mc}{\mathcal}

\newcommand{\mbb}{\mathbb}
\DeclareMathOperator{\rank}{rank}
\DeclareMathOperator{\spr}{sprank}

\makeatletter

\newcommand{\be}{\begin{equation}}
\newcommand{\ee}{\end{equation}}

\newcommand{\bc}{\begin{center}}
\newcommand{\ec}{\end{center}}
\newcommand{\bfl}{\begin{flushleft}}
\newcommand{\efl}{\end{flushleft}}

\newcommand{\beqa}{\begin{eqnarray}}
\newcommand{\eeqa}{\end{eqnarray}}
\newcommand{\beqan}{\begin{eqnarray*}}
\newcommand{\eeqan}{\end{eqnarray*}}
\newcommand{\beq}{\begin{equation}}
\newcommand{\eeq}{\end{equation}}

\begin{document}
% Title.
% ------
{\huge\title{Minimum Sparsity of Unobservable \\Power Network Attacks}} % Cyber-Physical

\author{%~\\
\IEEEauthorblockN{Yue Zhao, Andrea Goldsmith, and H. Vincent Poor}
%}

%\author{\IEEEauthorblockN{Yue Zhao\IEEEauthorrefmark{1}\IEEEauthorrefmark{2},
%Andrea Goldsmith\IEEEauthorrefmark{1} and
%H. Vincent Poor\IEEEauthorrefmark{2}}\\
%\IEEEauthorblockA{\IEEEauthorrefmark{1}Dept. of Electrical Engineering,
%Stanford University, Stanford, CA, 94305}\\
%\IEEEauthorblockA{\IEEEauthorrefmark{3}Dept. of Electrical Engineering, Princeton University, Princeton, NJ, 08540
%}
\thanks{This research was supported in part by the DTRA under Grant HDTRA1-08-1-0010, in part by the Air Force Office of Scientific Research under
MURI Grant FA9550-09-1-0643, and in part by the Office of Naval Research under Grant N00014-12-1-0767.}%
\thanks{Y. Zhao is with the Dept. of Electrical and Computer Engineering, Stony Brook University, Stony Brook, NY, 11794 USA (e-mail: yue.zhao.2@stonybrook.edu).}%
\thanks{A. Goldsmith is with the Dept. of Electrical Engineering, Stanford University, Stanford, CA 94305 USA (e-mail: andrea@ee.stanford.edu).}%
\thanks{H. V. Poor is with the Dept. of Electrical Engineering, Princeton University, Princeton, NJ 08544 USA (e-mail: poor@princeton.edu).}%
}

%\ninept
%
\maketitle
%\markboth{{\tiny Submitted to IEEE Transactions on Automatic Control, Special Issue on Control of Cyber-Physical Systems}}{}

\begin{abstract}
%\boldmath
Physical security of power networks under power injection attacks that alter generation and loads is studied. The system operator employs Phasor Measurement Units (PMUs) for detecting such attacks, while attackers devise attacks that are \emph{unobservable} by such PMU networks. %A linearized DC power flow model is employed for analytical tractability. %Firstly, for any set of PMUs in place, it is shown that the \emph{existence} of an unobservable attack that is restricted to any given subset of the buses can be determined with probability one based solely on the network topology. Next,
%For %the NP-hard problem of 
%finding the \emph{sparsest} unobservable attacks, with generic grid parameters, 
It is shown that, given the PMU locations, the solution to finding the sparsest unobservable attacks has a simple form with probability one, namely, $\kappa(\mc{G^M})+1$, where $\kappa(\mc{G^M})$ is defined as the %the so-defined 
vulnerable vertex connectivity of an augmented graph. The constructive proof allows one to find the entire set of the sparsest unobservable attacks in polynomial time. Furthermore, %the graph-theoretic interpretation of unobservable attacks leads to a natural characterization 
a notion of the potential impact of unobservable attacks is introduced. %Accordingly, a greedy algorithm for placing PMUs is developed to simultaneously raise the minimum sparsity as well as mitigate the maximum potential impact among all unobservable attacks. 
With optimized PMU deployment, 
the sparsest unobservable attacks and their potential impact as functions of the number of PMUs are evaluated numerically for the IEEE 30, 57, 118 and 300-bus systems and the Polish 2383, 2737 and 3012-bus systems. It is observed that, as more PMUs are added, the maximum potential impact among all the sparsest unobservable attacks drops quickly until it reaches the minimum sparsity.
\end{abstract}

\section{Introduction}
Modern power networks are increasingly dependent on information technology in order to achieve higher efficiency, flexibility and adaptability \cite{GuestEd}. The development of more advanced sensing, communications and control capabilities for
power grids enables better situational awareness and smarter control. However, security issues also arise as more complex information systems become prominent targets of cyber-physical attacks: not only can there be data attacks on measurements that disrupt situation awareness \cite{Liusec}, but also control signals of %many
power grid components including generation and loads can be hijacked, leading to immediate physical misbehavior of power systems \cite{Leonload}. Furthermore, in addition to %instead of trying to hack
hacking control messages, a powerful attacker can also implement physical attacks by directly intruding upon power grid components. Therefore, to achieve reliable and secure operation of a smart power grid, it is essential for the system operator to minimize (if not eliminate) the feasibility and impact %viability
of physical attacks.

There are many closely related techniques that can help achieve secure power systems. Firstly, coding and encryption can better secure control messages and communication links \cite{Liangsec}, and hence raise the level of difficulty of cyber %and physical
attacks. %However, not only can encryption be deciphered, but also such secrecy schemes do not prevent direct physical attacks. 
To prevent physical attacks, grid hardening is another design choice \cite{harden}. However, grid hardening can be very costly, and hence may only apply to a small fraction of the components in large power systems. Secondly, power systems are subject to many kinds of faults and outages \cite{kezunovic2011smart, ZGP12, ZSRGP13}, which are in a sense \emph{unintentional} physical attacks. As such outages are not inflicted by attackers, they are typically modeled as random events, and detecting outages is often modeled as a hypothesis testing problem \cite{Tajergrid}. However, this event and detection model is not necessarily accurate for \emph{intentional} physical attacks, which are the focus of this paper. Indeed, an intelligent attacker would often like to strategically \emph{optimize} its attack, such that it is not only hard to detect, but also the most viable to implement (e.g., with low execution complexity as well as high impact). %From this perspective, physical attacks can be viewed as (if not stronger than) the \emph{worst-case} outages in terms of the level of difficulty for detection.

Recently, there has been considerable research concerning data injection attacks on sensor measurements %, particularly
from supervisory control and data acquisition (SCADA) systems. A common and important goal among these works is to pursue the integrity of network \emph{state estimation}, that is, to successfully detect the injected data attack and recover the correct system states. The feasibility of constructing %false
data injection attacks to pass bad data detection schemes and alter estimated system states was first shown in \cite{Liusec}. There, a natural question arises as to how to find
the \emph{sparsest unobservable} data injection attack, as sparsity is used to model the complexity of an attack, as well as the resources needed for an attacker to implement it. However, finding such an \emph{optimal attack} requires solving an NP-hard $l_0$ minimization problem. %To tackle this hardness,
While efficiently finding the sparsest unobservable attacks in general remains an open problem, interesting and exact solutions under some special %different
problem settings have been developed in \cite{Kosut11} \cite{Souexact} \cite{Hexact}. %A recent result showed that the exact solution can be solved via $l_1$ relaxation under the assumption that the state estimator depends only on \emph{line flow} measurements \cite{Souexact}. It was also shown that the exact solution has a generalized min cut formulation under the assumption that the state estimator has \emph{full measurements} of the network, namely, all the line flows and power injections in the network are monitored \cite{Hexact}.
Another important aspect of a data injection attack is its impact on the power system. As state estimates are used to guide system and market operation of the grid, several interesting studies have investigated the impact of data attacks on optimal power flow recommendation \cite{teixeira12} and location marginal prices in a deregulated power market \cite{Liyan14} \cite{Jinsub}.
Furthermore, as Phasor Measurement Units (PMUs) become increasingly deployed in power systems, network situational awareness for grid operators is significantly improved compared to using legacy SCADA systems only. However, while PMUs provide accurate %and synchronous,
and secure sampling of the system states, their high installation costs prohibit ubiquitous deployment. Thus, the problem of how to economically deploy PMUs such that the state estimator can best detect data injection attacks is an interesting problem that many studies have addressed (see, e.g. \cite{Giani11, Tungattack, bobba10} among others.)

Compared to data attacks that target state estimators, physical attacks that directly disrupt power network physical processes can have a much faster %(and often stronger) %damage
impact on power grids. In addition to physical attacks by hacking control signals or directly intruding upon grid components, several types of load altering attacks have been shown to be practically implementable via Internet-based message attacks \cite{Leonload}. Topological attacks are another type of physical attack which have been considered in \cite{Kartopo}. %there, however, the attacks are modeled similarly to random outages as opposed to strategically optimized attacks.
Dynamic power injection attacks have also been analyzed in several studies. For example, in \cite{Florian}, conditions for the existence of undetectable and unidentifiable attacks were provided, and the sizes of the sets of such attacks were shown to be bounded by graph-theoretic quantities. %However, sparsity minimization of attacks are not discussed.
Alternatively, in \cite{TabuadaJ} and \cite{TabuadaC}, state estimation is considered in the presence of both power injection attacks and data attacks. Specifically, in these works, the maximum number of attacked nodes that still results in %can be tolerated for
correct estimation was characterized, and effective heuristics for state recovery under sparse attacks were provided.

In this paper, we investigate a specific type of physical attack in power systems called \emph{power injection attacks}, that alter generation and loads in the network. %Given the nonlinearity of the AC power flow model, for analytical tractability, 
A linearized power network model - the DC power flow model - is employed for simplifying the analysis of the problem and obtaining a simple solution that yields considerable insight.  %A DC power flow model is employed. 
We consider a grid operator that employs PMUs to (partially) monitor the network for detecting power injection attacks. Since power injection attacks disrupt the power system states immediately, the timeliness of PMU measurement feedback is essential. %the system model is generalized to allow 
Furthermore, our model allows for the  power injections at some buses to be ``unalterable''. This captures the cases of ``zero injection buses'' with no generation and load, and buses that are protected by the system operator. 
%We first address the feasibility problem of unobservable attacks restricted to any subset of buses. In this context, we prove that the existence of a power injection attack restricted to a given set of buses and yet unobservable by a given set of PMUs can be determined with probability one \emph{solely based on the topology of the power network}. Then,
Under this model we study the open $l_0$ minimization problem of finding the sparsest unobservable attacks given any set of PMU locations. We %first
start with a feasibility problem for unobservable attacks. We prove that the existence of an unobservable power injection attack restricted to any given set of buses can be determined with probability one by computing a quantity called the structural rank. Next,
we prove that the NP-hard problem of finding the sparsest unobservable attacks has a simple solution with probability one. Specifically, the sparsity of the optimal solution is $\kappa(\mc{G^M})+1$, where $\kappa(\mc{G^M})$ is the ``vulnerable vertex connectivity'' that we define for an augmented graph of the original power network. Meanwhile, the entire set of globally optimal solutions (there can be many of them) is found in polynomial time. We further introduce a notion of %Furthermore, the graph-theoretic interpretation of these sparsest unobservable attacks leads to a natural characterization of their 
potential %damage
impacts of unobservable attacks. Accordingly, among all the sparsest unobservable attacks, an attacker can easily find the one with the greatest potential impact. % in polynomial time
%Then, we develop a %greedy
%PMU placement algorithm that raises the minimum sparsity as well as mitigates the maximum potential impact among all unobservable attacks. 
Finally, given optimized PMU placement, %for all possible numbers of PMUs with optimized placement,
we evaluate the sparsest unobservable attacks in terms of their sparsity and potential impact in the IEEE 30, 57, 118 and 300-bus, and the Polish 2383, 2737 and 3012-bus systems.

The remainder of the paper is organized as follows. In Section \ref{secform}, models of the power network, power injection attacks, PMUs and unalterable buses are established. In addition, the minimum sparsity problem of unobservable attacks is formulated. In Section \ref{secfeas} we provide the feasibility condition for unobservable attacks restricted to any subset of the buses. In Section \ref{secmin} we prove that the minimum sparsity of unobservable attacks can be found in polynomial time with probability one. %The potential impact of the sparsest unobservable attacks is characterized based on a graph-theoretic interpretation. 
In Section \ref{secnum}, a %efficient
PMU placement algorithm for countering power injection attacks is developed, and numerical evaluation of the sparsest unobservable attacks in IEEE benchmark test cases and large-scale Polish power systems are provided. Conclusions are drawn in Section \ref{secconc}.

\section{Problem formulation} \label{secform}
\subsection{Power network model}
We consider a power network with $N$ buses, and denote the set of buses and the set of transmission lines by $\mc{N} = \{1,2,\ldots,N\}$ and $\mc{L} = \{1,2,\ldots,L\}$ respectively. For a line $\l\in\mc{L}$ that connects buses $n$ and $m$, denote its reactance by $x_l$ as well as $x_{nm}$, and define its \emph{incidence vector} $\bm{m}_l$ as follows:
\begin{align}
\bm{m}_l(i) = \left\{
\begin{array}{lll}
1, &\text{if } i=n, \\
-1, &\text{if } i=m, \\
0, &\text{otherwise}.
\end{array}
\right. \nn%\label{mcolumn}
\end{align}
Based on the power network topology and line reactances, we construct a weighted graph $\mc{G} = \{\mc{N}, \mc{L}, \bm{w}\}$ where the edge weight $w_l\triangleq \frac{1}{x_l}, \forall l = 1,\ldots,L$. 
The power system is generally modeled by nonlinear AC power flow equations \cite{WW96}. In this paper, a linearized model - the DC power flow model - is employed as an approximation of the AC model, which allows us to find a simple closed-form solution to the problem from which we glean significant insights. %helps us gain analytical insight of the problem. 
Under the DC model, 
the real power injections $\bm{P}\in\mbb{R}^N$ and the voltage phase angles $\bm{\theta}\in\mbb{R}^N$ satisfy $\bm{P} = \bm{B}\bm{\theta}$, %the following equation:
%\begin{align}
%$\bm{P} = \bm{B}\bm{\theta}, \label{DCmod}$
%\end{align}
where $\bm{B} = \sum_{l=1}^L \frac{1}{x_l}\bm{m}_l\bm{m}_l^T \in\mbb{R}^{N\times N}$
%\begin{align}
%\bm{B} = \sum_{l=1}^L \frac{1}{x_l}\bm{m}_l\bm{m}_l^T \in\mbb{R}^{N\times N} \label{lapgr} %= MD_xM^T
%\end{align}
is the \emph{Laplacian} of the weighted graph $\mc{G}$.
We assume that $x_l$ is positive which is typically true for transmission lines (cf. Chapter 4 of \cite{WW96}).
Furthermore, the power flow on line $l$ from bus $n$ to bus $m$ equals $P_{nm} = \frac{1}{x_{nm}}(\theta_n - \theta_m)$. %is computed by
%\begin{align}
%P_{nm} = \frac{1}{x_{nm}}(\theta_n - \theta_m). \label{pfline}
%\end{align}

We consider %In a cyber-physical attack, by attacking the control signals of generation and loads, an attacker can implement a
attackers inflicting power injection attacks that alter the generation and loads in the power network. We denote the power injections in normal conditions by $\bm{P}$, and denote a power injection attack by $\Delta \bm{P}\in\mbb{R}^N$. Thus the post-attack power injections are $\bm{P}+ \Delta \bm{P}$. 

%\textcolor{red}{add attack purpose}

%\begin{RK}
%\end{RK}

\subsection{Sensor model and unobservable attacks} \label{secsensemod}
We consider the use of PMUs by the system operator for monitoring the power network in order to detect power injection attacks. With PMUs installed at the buses, we consider the following two different sensor models:
\begin{enumerate}
\item A PMU securely measures the voltage phasor of the bus at which it is installed.\footnote{The voltage phase angles at all the buses are defined to be relative to a common reference --- the phase angle at the angle %n angle 
reference bus in the network.}
\item A PMU securely measures the voltage phasor of the bus at which it is installed, as well as the current phasors on all the lines connected to this bus\footnote{In practice, the second PMU measurement model is achieved by installing PMUs on all the lines connected to a bus.}.
\end{enumerate}
We denote the set of buses with PMUs by $\mc{M}~(\subseteq\mc{N})$, and let $M\triangleq|\mc{M}|$ be the total number of PMUs, where $|\cdot|$ denotes the cardinality of a set. Without loss of generality (WLOG), we choose one of the buses in $\mc{M}$ to be the angle reference bus. 
We say that a power injection attack $\Delta\bm{P}$ is \emph{unobservable} if it leads to \emph{zero} changes in all the quantities measured by the PMUs. With the first PMU model described above, we have the following definition: %the unobservability condition of an attack is defined as follows,
\begin{define}[Unobservability condition]
An %power injection
attack $\Delta \bm{P} \ne 0$ is unobservable if and only if
\begin{align}
\exists \Delta\bm{\theta}, ~\text{such that}~ \Delta\bm{P} = \bm{B}\Delta\bm{\theta} \text{ and } \Delta\bm{\theta}_{\mc{M}}=\bm{0}, \label{unobcon}
\end{align}
where $\Delta\bm{\theta}_\mc{\mc{M}}$ denotes the $M\times1$ sub-vector of $\Delta\bm{\theta}$ obtained by keeping its $M$ entries whose indices are in $\mc{M}$.\footnote{Since $\bm{B}$ is a weighted Laplacian matrix, the elements of $\Delta \bm{P}$ sum to $0$.}
\end{define}
With the second PMU model described above, for any bus $n\in\mc{N}$, it is immediate to verify that the following three conditions are equivalent:
\begin{enumerate}
\item There are no changes of the voltage phasor at $n$ and of the current phasors on all the lines connected to $n$.
\item There are no changes of the voltage phasor at $n$ and of the power flows on all the lines connected to $n$.
\item $\forall n'\in N[n]$, there is no change of the voltage phasor at $n'$, where $N[n]$ is the closed neighborhood of $n$ which includes $n$ and its neighboring buses $N(n)$.
\end{enumerate}
%Therefore,
Thus, for forming unobservable attacks, the following two situations are equivalent to the attacker:
\begin{itemize}
\item The system operator monitors the set of buses $\mc{M}$ with the second PMU model;
\item The system operator monitors the set of buses $N[\mc{M}]$ with the first PMU model,
\end{itemize}
where $N[\mc{M}]$ is the closed neighborhood of $\mc{M}$ which includes all the buses in $\mc{M}$ and their neighboring buses $N(\mc{M})$.
Thus, the unobservability condition with the second PMU model is obtained by replacing $\mc{M}$ with $N[\mc{M}]$ in \eqref{unobcon}. %and the second PMU model is mathematically a special case of the first PMU model. %the one that
%We see that
%Without loss of generality (WLOG), 
WLOG, we employ the first PMU model in the following analysis, and based on the discussion above all the results can be directly translated to the second PMU model. %In the numerical evaluation section, we employ the second PMU model

%From an attacker's interest, it would like to pursue unobservability in order to have successful attacks without being detected. This motivates us to study optimization of attacks that satisfy the unobservability condition in the following.

\subsection{Sparsest unobservable attacks} \label{secmM}
%When implementing an attack, an attacker has three general objectives:
%three-way trade-off
In forming an unobservable attack, an attacker generally has two objectives: minimize execution complexity and maximize its impact on the grid. Note that these two objectives can be competing interests that are not simultaneously achievable. We will first focus on finding the minimum execution complexity for an attack to be unobservable, which constitutes the main part of this work. Among attacks with the minimum complexity, we then find the one with the maximum impact.

For an attack vector $\Delta\bm{P}$, we use its zero norm $\Vert\Delta\bm{P}\Vert_0$ to model its execution complexity. This is because attackers are typically resource-constrained, and can choose only a limited number of buses to implement attacks. For minimizing attack complexity, an attacker is interested in finding the sparsest attacks that satisfy the unobservability condition \eqref{unobcon}: %following $l_0$ minimization problem based on
\begin{align}
    \min_{\Delta\bm{\theta}}~ &\Vert\Delta\bm{P}\Vert_0 \label{l0min}\\
    s.t.~& \Delta\bm{P} = \bm{B\Delta\theta}, ~\Delta\bm{\theta}_{\mc{M}} = \bm{0}, ~\Delta\bm{\theta} \ne \bm{0}. \nn
\end{align}
Since $\Delta\bm{\theta}_{\mc{M}} = \bm{0}, ~\Delta\bm{\theta} \ne \bm{0} \Rightarrow \bm{B\Delta\theta} = \bm{B}_{\mc{N}\mc{M}^c}\Delta\bm{\theta}_{\mc{M}^c}, \Delta\bm{\theta}_{\mc{M}^c}\ne \bm{0}$, %equivalently, 
a more compact form of \eqref{l0min} is as follows:
\begin{align}
    \eqref{l0min} ~\Leftrightarrow~ &\min_{\Delta\bm{\theta}_{\mc{M}^c}\ne \bm{0}}~ \Vert\bm{B}_{\mc{N}\mc{M}^c}\Delta\bm{\theta}_{\mc{M}^c}\Vert_0, \label{l0min_2}
\end{align}
where $\mc{M}^c = \mc{N}\backslash\mc{M}$ denotes the complement of $M$, and $\bm{B}_{\mc{N}\mc{M}^c}$ is the submatrix of $\bm{B}$ formed by choosing all its rows and a %subset
set of columns $\mc{M}^c$.

We now note that problem \eqref{l0min_2} is NP-hard: Specifically, % that Because of the non-convexity of the $l_0$ norm, problem \eqref{l0min_2} is in general NP-hard. 
%First, 
%\eqref{l0min_2} is called the \emph{cospark} of the matrix $\bm{B}_{\mc{N}\mc{M}^c}$ \cite{CTLP}, and it has been proven that computing the cospark of a matrix is NP-hard \cite{CScompute12}. Due to the Laplacian structure of $\bm{B}$, 
as a special case of the cospark problem of a matrix \cite{CTLP} problem \eqref{l0min_2} resembles a %is equivalent to a %resembles the 
security index problem discussed in \cite{Hexact}, which has been proven to be NP-hard. 
Under some special problem settings for data injection attacks, problems of this type have been shown to be solvable exactly in polynomial time \cite{Kosut11} \cite{Souexact} \cite{Hexact}. In general, low complexity heuristics have been developed for solving $l_0$ minimization problems (e.g., $l_1$ relaxation). 
%It is worth noting that, in the field of compressed sensing, theoretical optimality guarantees of certain low complexity heuristics have been developed based on the desirable properties (e.g., the restricted isometry property) of the designed sensing matrix \cite{CT06}. %which is typically generated with independent and identically distributed random entries \cite{CT06}.
%However, such theoretical guarantees do not apply to our problem setting, because the matrix $\bm{B}_{\mc{N}\mc{M}^c}$ is a submatrix of the Laplacian of the underlying grid, and typically does not enjoy the properties of the %designed
%sensing matrix in the context of compressed sensing.

We now generalize our model to allow a subset of buses to be ``unalterable buses'', meaning that their nodal power injection cannot be changed by attackers. This allows us to model the following scenarios:
\begin{itemize}
\item A ``zero injection'' bus that simply connects multiple lines without nodal generation or load, and hence its power injection is always zero and cannot be changed. 
\item A ``protected'' bus by the system operator, and its power injection is not accessible by the attacker. 
\end{itemize} 
We denote the set of unalterable buses by $\mc{U}$. %In practice there may be good reasons for $\mc{M}\subset\mc{U}$. In this paper, we stay general and do not assume specific 
The other buses $\mc{U}^c$ are termed ``alterable'' buses. 
Generalizing \eqref{l0min}, the sparsest unobservable attack problem is established as follows: 
\begin{align}
    \min_{\Delta\bm{\theta}}~ &\Vert\Delta\bm{P}\Vert_0 \label{l0min_g}\\
    s.t.~& \Delta\bm{P} = \bm{B\Delta\theta}, ~\Delta\bm{\theta}_{\mc{M}} = \bm{0}, ~\Delta\bm{P}_{\mc{U}} = \bm{0}, ~\Delta\bm{\theta} \ne \bm{0}. \nn
\end{align}
When $\mc{U} = \emptyset$, \eqref{l0min_g} reduces to \eqref{l0min}. Generalizing \eqref{l0min_2}, Eq. \eqref{l0min_g} has the following equivalent form: 
\begin{align}
    \eqref{l0min_g} ~\Leftrightarrow~ &\min_{\substack{\Delta\bm{\theta}_{\mc{M}^c}\ne \bm{0}, \\\left(\bm{B}_{\mc{N}\mc{M}^c}\Delta\bm{\theta}_{\mc{M}^c}\right)_{\mc{U}}=0}}~ \Vert\bm{B}_{\mc{N}\mc{M}^c}\Delta\bm{\theta}_{\mc{M}^c}\Vert_0. \label{l0min_g2}
\end{align}

\subsection{Graph augmentation} \label{sec:aug}
\emph{Given the locations of the sensors} $\mc{M}$, we now introduce a variation of the graph $\mc{G}$ that will prove key to developing the main results later. 

\begin{define} \label{gmdef}
Given a set of buses $\mc{M}\subseteq\mc{N}$, $\mc{G^M}$ is defined to be the following augmented graph based on $\mc{G}$: 
\begin{enumerate}
\item $\mc{G^M}$ includes all the buses in $\mc{G}$, and has one additional unalterable dummy bus. 
\item Define an augmented set $\bar{\mc{M}}$ that contains $\mc{M}$ and the unalterable dummy bus. 
\item $\mc{G^M}$ includes all the edges of $\mc{G}$, and an edge is added between every pair of buses in $\bar{\mc{M}}$, 
and its weight can be chosen arbitrarily as any positive number. 
\end{enumerate}%In addition to all the existing edges of $\mc{G}$, an edge is added between every pair of buses in $\mc{M}$, and its weight can be chosen arbitrarily as any positive number. %its weight an independent continuous random variable strictly bounded away from zero from below.
\end{define}
%We emphasize that, whenever we use the augmented graph $\mc{G^M}$, the set of sensor locations $\mc{M}$ includes the unalterable dummy bus. %
We note that the dummy bus is only connected to the set of sensors $\mc{M}$. 
We observe the following key facts. First, an unobservable attack in the original graph $\mc{G}$ leads to zero changes in all the voltage phase angles in $\mc{M}$. Thus, any line between a pair of buses in $\mc{M}$ would see a zero change of the power flow on it. It is then clear that %applying an unobservable attack to $\mc{G^M}$ does not lead to 
%As a result, 
the added dummy bus and lines in $\mc{G^M}$ do not lead to any power flow changes in the network under any unobservable attack. 
We thus have the following lemma:
\begin{lem} \label{lemaug}
An attack is unobservable by $\mc{M}$ in $\mc{G}$ if and only if it is unobservable by $\mc{M}$ %\footnote{An abuse of notation: the set of sensor locations $\mc{M}$ in the augmented graph $\mc{G^M}$ includes the added dummy bus.} 
in $\mc{G^M}$. 
\end{lem}
This allows us to work with the augmented graph $\mc{G^M}$ instead of $\mc{G}$. It is clear that the weights of the added edges in $\mc{G^M}$ do not matter for Lemma \ref{lemaug} to hold. 

%Instead of applying existing solutions,
%{\color{red} [to change]}

%In this paper, we assume generic grid parameters as detailed in Section \ref{secfeas}. We will start with a feasibility problem of unobservable attacks, whose solution will be useful in solving \eqref{l0min_2}. Next, we will solve \eqref{l0min_2} in Section \ref{secmin} by exploiting the structure of the Laplacian matrix $\bm{B}$ and the graph-theoretic interpretation of unobservable attacks. %From the developed insight of the sparsest unobservable attacks,
%We will then characterize the potential impact associated with %an
%unobservable attacks. %The optimal solution leads to To start with, we first study the feasibility condition of unobservable attacks in the next section.

\section{Feasibility condition of unobservable attacks} \label{secfeas}
In this section, we address the following question whose solutions will be useful in solving the minimum sparsity problem \eqref{l0min_g2}: Assuming that the attacker can only alter the power injections at a subset of the buses, denoted by $\mc{A}\subseteq\mc{U}^c$, does there \emph{exist} an attack that is unobservable by a set of PMUs $\mc{M}$? For any given $\mc{A}$, a feasible non-zero attack $\Delta\bm{P}~(\ne\bm{0})$ must satisfy $\Delta\bm{P}_{\mc{A}^c} = \bm{0}$. %where $\mc{A}^c = \mc{N}\backslash\mc{A}$. 
In other words, it must not alter the power injections at the buses in $\mc{A}^c$. %Such feasibility problems are of interest to an attacker as it may have limited access to only a subset of power grid components.

From \eqref{unobcon}, there exists an unobservable non-zero attack if and only if
\begin{align}
&\exists \Delta \bm{P}, \Delta\bm{\theta}\ne\bm{0},~s.t.\nn\\
&~~~~\Delta\bm{P} = \bm{B}\Delta\bm{\theta}, ~\Delta\bm{P}_{\mc{A}^c} = \bm{0}, ~\Delta\bm{\theta}_{\mc{M}}=\bm{0}.  \label{unobconprf1}
\end{align}
Since $\left\{
\begin{array}{lll}
\Delta\bm{\theta}_{\mc{M}}=\bm{0} \\
\Delta\bm{\theta}\ne\bm{0}
\end{array}
\right. \Rightarrow \Delta\bm{\theta}_{\mc{M}^c}\ne\bm{0},\Delta \bm{P}\ne\bm{0}$, we have that \eqref{unobconprf1} is equivalent to
\begin{align}
\exists \Delta\bm{\theta}_{\mc{M}^c}\ne\bm{0}, ~s.t.~(\Delta\bm{P}_{\mc{A}^c} = )~\bm{B}_{\mc{A}^c\mc{M}^c}\Delta\bm{\theta}_{\mc{M}^c} = \bm{0}, \label{unobconprf2}
\end{align}
where $\bm{B}_{\mc{A}^c\mc{M}^c}$ is the submatrix of $\bm{B}$ formed by its rows $\mc{A}^c$ and columns $\mc{M}^c$. % i.e., $\bm{B}_{\mc{A}^c\mc{M}^c}$ is column rank deficient.
An illustration of \eqref{unobconprf2} is depicted in Figure \ref{feasfig}, where the submatrix formed by the shaded blocks represents $\bm{B}_{\mc{A}^c\mc{M}^c}$. %Therefore,
From \eqref{unobconprf2}, we have the following lemma on the feasibility condition of unobservable attacks.
\begin{lem} \label{lemfeas1}
Given $\mc{A}$ and $\mc{M}$, there exists an unobservable non-zero attack if and only if $\bm{B}_{\mc{A}^c\mc{M}^c}$ is column rank deficient.
\end{lem}

\begin{figure}[tb]
  \centering
  \includegraphics[scale = 0.75]{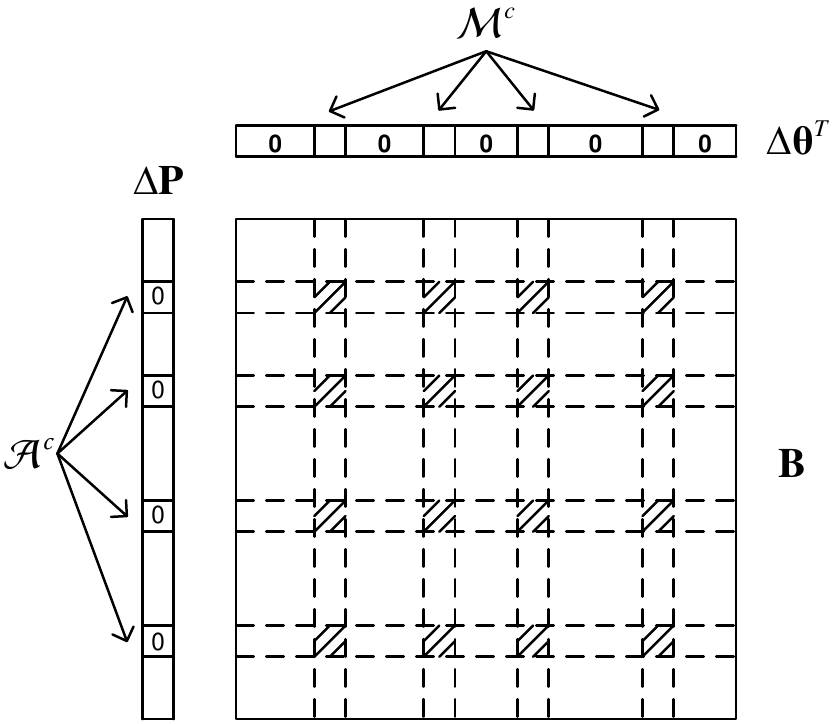}
  \caption{An illustration of \eqref{unobconprf2} where the submatrix formed by the shaded blocks represents $\bm{B}_{\mc{A}^c\mc{M}^c}$.}
  \label{feasfig}
\end{figure}

%In Section \ref{secfeas}, we will show that whether or not $\bm{B}_{\mc{A}^c\mc{M}^c}$ has full column rank can be determined almost surely from the topology of the power network. %Moreover, we will see that feasibility conditions plays an important role

%In this section, we study the feasibility condition of unobservable attacks restricted to a subset of the buses denoted by $\mc{A}$.
%Given a set of PMUs $\mc{M}$, from Lemma \ref{lemfeas1}, there exists an unobservable attack if and only if $\rank\left(\bm{B}_{\mc{A}^c\mc{M}^c}\right) < |\mc{M}^c|$.
To analyze when this column rank deficiency condition, $\rank\left(\bm{B}_{\mc{A}^c\mc{M}^c}\right) < |\mc{M}^c|$, is satisfied, we start with the following observations based on the fact that $\bm{B}$ is the Laplacian of the weighted graph $\mc{G}$.
\begin{enumerate}
\item The \emph{signs} ($+1$, $-1$, or $0$) of the entries of $\bm{B}$ are fully determined by the \emph{network topology}:
    \begin{align}
    \bm{B}_{ij} > 0, &\text{ if } i=j, \nn\\
    \bm{B}_{ij} < 0, &\text{ if node (bus) $i$ and node $j~(i\ne j)$} \nn\\ 
    \text{are}~&\text{connected by an edge (transmission line)}, \nn\\
    \bm{B}_{ij} = 0, &\text{ if node (bus) $i$ and node $j~(i\ne j)$ are} \nn\\ 
    \text{not}~&\text{connected}.\nn
    \end{align}
\item The \emph{values} of the non-zero entries of $\bm{B}$ are determined by the \emph{line reactances} $\{x_{ij}\}$:%edge weights $\bm{w}$ of the graph $\mc{G}$:
\begin{align}
\bm{B}_{ii} &= \sum_{j\ne i} w_{ij} = \sum_{j\ne i} \frac{1}{x_{ij}}, \nn\\
\bm{B}_{ij} &= -w_{ij} = -\frac{1}{x_{ij}}, \text{ if } i\ne j \text{ and } \bm{B}_{ij}\ne0. \nn
\end{align}
\end{enumerate}
When all the line reactances in the power network are known, so are the entries of the submatrix $\bm{B}_{\mc{A}^c\mc{M}^c}$, %are determined,
and it is immediate to compute whether $\rank\left(\bm{B}_{\mc{A}^c\mc{M}^c}\right)<|\mc{M}^c|$. \emph{Without} knowing the exact values of any line reactances, %the entries in $\bm{B}_{\mc{A}^c\mc{M}^c}$,
we will show %in the following
that whether $\rank\left(\bm{B}_{\mc{A}^c\mc{M}^c}\right)<|\mc{M}^c|$ can be determined almost surely by computing the \emph{structural rank} of $\bm{B}_{\mc{A}^c\mc{M}^c}$, defined as follows \cite{sprankbook}.

\begin{define}[Set of independent entries]\label{indent}
A set of independent entries of a matrix $\bm{H}$ is a set of nonzero entries, no two of which lie on the same line (row or column).
\end{define}
\begin{define}[Structural rank]
The structural rank of a matrix $\bm{H}$, denoted by $\spr(\bm{H})$, is the maximum number of elements contained in at least one set of independent entries.
\end{define}

A basic relation between the structural rank and the rank of a matrix is the following \cite{sprankbook},
\begin{align}
\spr(\bm{B}_{\mc{A}^c\mc{M}^c}) \ge \rank(\bm{B}_{\mc{A}^c\mc{M}^c}).
\end{align}
In the literature, structural rank is also termed ``generic rank'' \cite{grank}.

Specifically, we consider \emph{generic} power grid parameters, i.e., we assume that the line reactances $x_l~(l=1,2,\ldots,L)$ are independent, but not necessarily identical random variables drawn from continuous probability distributions. We assume that the reactances are bounded away from zero from below
%is a continuous distribution strictly bounded above from zero
(as lines do not have zero reactances in practice). As such, the analysis in this work is along the line of \emph{structural properties} as in \cite{sprankbook} and \cite{grank}, and we will develop results that hold \emph{with probability one}. We believe the independence (but not identically distributed) assumption is sufficiently general in practice. In particular, there are uncertainties in factors that influence the reactance of a line (e.g. the distance that a line travels, the degradation of a line over time). These uncertainties can be modeled as independent (but not identically distributed) random variables, leading to the model employed in this paper. 

Clearly, $\bm{B}_{\mc{A}^c\mc{M}^c}$ is always column rank deficient when $|\mc{A}^c| < |\mc{M}^c|$. Next, we discuss the case of $|\mc{A}^c| \ge |\mc{M}^c|$. %and then extend it to the case of $|\mc{A}^c| > |\mc{M}^c|$.
%\subsection{The case of $|\mc{A}^c| = |\mc{M}^c|$}
%In this case, $\bm{B}_{\mc{A}^c\mc{M}^c}$ is a square matrix.
We begin with the special case $\mc{A} = \mc{M}$, for which we have the following lemma whose proof is relegated to Appendix \ref{secprflemL}:
\begin{lem} \label{lemL}
Let $\bm{B}\in\mbb{R}^{N\times N}$ be the Laplacian of a connected %weighted
graph $\mc{G}$ with strictly positive edge weights. For any set of node indices $\mc{I}\subset\{1,2,\ldots,N\}$, denote by $\bm{B}_\mc{II}$ the submatrix of $\bm{B}$ formed by its rows $\mc{I}$ and columns $\mc{I}$. Then $\forall \mc{I}, |\mc{I}| \le N-1$, $\bm{B}_\mc{II}$  is of full rank. %, ~|\mc{I}|\triangleq I \le N-1
\end{lem}

Note that Lemma \ref{lemL} holds deterministically without assuming generic edge weights of the graph.
For the case of $\mc{A} = \mc{M}$, we let $\mc{I} = \mc{A}^c = \mc{M}^c$, and Lemma \ref{lemL} proves that $\rank\left(\bm{B}_{\mc{A}^c\mc{M}^c}\right) = |\mc{M}^c|$. This implies the intuitive fact that there exists no attack restricted to $\mc{A}$ that is unobservable by a set of PMUs $\mc{M}=\mc{A}$.

Now, we address the general case of arbitrary $\mc{A}$ and $\mc{M}$. %in which $\mc{A}$ and $\mc{M}$ may not be the same.
%We first define the following property of a matrix, which will be shown to be essential for a matrix to be of full rank in our context.
%\begin{define}[The non-zero permuted diagonal property] \label{nonzerodiag}
%A matrix $\bm{H}\in\mbb{R}^{N\times N}$ has a non-zero permuted diagonal if and only if there exists a permutation function $\pi(i), i=1,2,\ldots,N$, such that $\bm{H}_{i,\pi(i)} \ne 0, \forall i=1,2,\ldots,N$.
%\begin{align}
%\exists
%\end{align}
%\end{define}
%Note that the above property is equivalent to the following: there exists a way of permuting the columns (or rows) of $\bm{H}$ such that the diagonal entries of the resulting matrix are all non-zero.
We have the following theorem demonstrating that having $\spr(\bm{B}_{\mc{A}^c\mc{M}^c}) =  |\mc{M}^c|$ \emph{almost surely guarantees} $\rank(\bm{B}_{\mc{A}^c\mc{M}^c}) =  |\mc{M}^c|$. The proof is relegated to Appendix \ref{secprffeasthm}.
\begin{thm} \label{feasthm}
For a connected weighted graph $\mc{G} = \{\mc{N},\mc{L},\bm{w}\}$, assume that the edge weights are independent continuous random variables strictly bounded away from zero from below, and denote the Laplacian of $\mc{G}$ by $\bm{B}\in\mbb{R}^{N\times N}$. Then, any $N'\times N''$ submatrix of $\bm{B}$, with $\min(N', N'') \le N-1$, has a rank of $\min(N', N'')$ %is of full rank
with probability one if it has a structural rank of $\min(N', N'')$.
\end{thm}

From Theorem \ref{feasthm}, with $|\mc{A}^c| \ge |\mc{M}^c|$, if $\spr(\bm{B}_{\mc{A}^c\mc{M}^c}) =  |\mc{M}^c| \le N-1$, %as long as $\bm{B}_{\mc{A}^c\mc{M}^c}$ is not the full matrix $\bm{B}$,
we have with probability one that $\rank(\bm{B}_{\mc{A}^c\mc{M}^c}) = |\mc{M}^c|$, and there exists no attack restricted to $\mc{A}$ that is unobservable by a set of PMUs $\mc{M}$. %We also have the next lemma (see e.g., \cite{sprankbook}) whose proof is immediate from the Leibniz formula for determinants. %whose proof is relegated to Appendix \ref{secprffeasconv}, we
%shows that .
%\begin{lem}\label{feasconv}
%$\spr(\bm{B}_{\mc{A}^c\mc{M}^c}) =  |\mc{M}^c|$ is a necessary condition for $\rank(\bm{B}_{\mc{A}^c\mc{M}^c}) =  |\mc{M}^c|$.
%\end{lem}

\begin{RK}
It has been known in the literature that (see e.g., \cite{sprankbook}), a full structural rank of a matrix leads to a full rank matrix with probability one, as long as the nonzero entries in the matrix are drawn independently from continuous probability distributions. However, it is worth noting that this is not sufficient for proving Theorem \ref{feasthm}. This is because, as in Theorem \ref{feasthm}, we are interested in matrices that are \emph{submatrices of a graph Laplacian}: even with the edge weights of the graph drawn independently, the entries in these submatrices are \emph{correlated} due to the special structure of a graph Laplacian. Such correlation leads to technical difficulties for the proof, which can be overcome as shown in Appendix \ref{secprffeasthm}.
\end{RK}

%Recall that, based solely on the network topology, one can determine the signs of the entries in $\bm{B}$, and hence whether or not $\bm{B}_{\mc{A}^c\mc{M}^c}$ has a non-zero permuted diagonal for given $\mc{A,M}$. Thus, we have the following corollary of Theorem \ref{feasthm}:
%\begin{cor}
%Given $\mc{A,M}$ with $|\mc{A}^c| = |\mc{M}^c|$, whether there exists a non-zero unobservable attack can be determined with probability one based solely on the knowledge of the network topology of the grid.
%\end{cor}

We note that the structural rank of a matrix can be computed in polynomial time by finding the maximum bipartite matching in a graph \cite{sprankbook}. Since whether an entry of $\bm{B}$ is non-zero is solely determined by the topology of the network, we have the following corollary.

\begin{cor}
Given $\mc{A}$ and $\mc{M}$, whether a non-zero unobservable attack exists can be determined with probability one based solely on the knowledge of the grid topology.
\end{cor}

\section{Solving the sparsest unobservable attacks} \label{secmin}
In this section, we study the problem of finding the sparsest unobservable attacks given any set of PMUs $\mc{M}$ (cf. \eqref{l0min_g2}). %with the minimum execution complexity, modeled as the \emph{sparsest} unobservable attacks \eqref{l0min_g2}.
As remarked in Section \ref{secmM}, $l_0$ minimization such as \eqref{l0min_g2} is NP-hard. % From another perspective, based on the feasibility condition from Section \ref{secfeas}, solving \eqref{l0min_g2} is equivalent to finding a subset of the buses $\mc{A}$ with the minimum cardinality such that there exists an unobservable attack restricted to it:
%\begin{align}
%\min_{\mc{A}\subseteq\mc{N}}~& |\mc{A}| \label{theQ}\\
%s.t.~& \det(\bm{B}_{\mc{A}^c\mc{M}^c}) \ne 0. \nn
%\end{align}
%Clearly, finding the optimal subset $\mc{A}$ via exhaustive search consumes an undesirable exponential complexity.
We will show that the sparsest unobservable attack can in fact be found in \emph{polynomial time with probability one}. We first introduce a key concept --- a \emph{vulnerable vertex cut}. We then state our main theorem that yields an explicit solution for the sparsest unobservable attack problem \eqref{l0min_g2}. We prove that this solution both upper and lower bounds the optimum of \eqref{l0min_g2}, hence proving the theorem. 
%prove an important role of the \emph{vertex connectivity} of the grid in lower bounding the optimum of \eqref{l0min_g2}. %we show constructively that such lower bound can actually be achieved.
%We then derive an upper bound of \eqref{l0min_g2}. By further exploiting the graph-theoretic insights behind the upper bound, we close the gap between the lower and upper bounds, and provide a complete solution to the sparsest unobservable attack problem \eqref{l0min_g2}. 
%Finally, we characterize the potential impact of unobservable attacks based on their graph-theoretic interpretations.

%\subsection{The role of vertex connectivity in lower bounding the sparsity of unobservable attacks}
\subsection{Vulnerable vertex cut and vulnerable vertex connectivity}
We start with the following basic definitions:
\begin{define}[Vertex cut]
A vertex cut of a connected graph $\mc{G}$ is a set of vertices whose removal renders $\mc{G}$ disconnected.
\end{define}
\begin{define}[Vertex connectivity]
The vertex connectivity of a graph $\mc{G}$, denoted by $\kappa(\mc{G})$, is the size of the minimum vertex cut of $\mc{G}$, i.e., it is the minimum number of vertices that need to be removed to make the remaining graph disconnected.
\end{define}

From the definition of the augmented graph $\mc{G^M}$ in Section \ref{sec:aug}, since all the buses in $\bar{\mc{M}}$ (containing $\mc{M}$ and the dummy bus) are pair-wise connected, we have the following lemma: 
\begin{lem} \label{auglem}
For any vertex cut of the augmented graph $\mc{G^M}$, there is no pair of the buses in $\bar{\mc{M}}$ that are disconnected by this cut.
\end{lem}
Accordingly, we introduce the following notations which will be used later on: 

%{\color{red}[to change]}
\begin{ntn} \label{ntn1}
Given a vertex cut of $\mc{G^M}$, we denote the set of buses disconnected from $\bar{\mc{M}}$ after removing the cut set by $\mc{S}$. The cut set itself is denoted by $N(\mc{S})$. 
\end{ntn}
With the vertex cut $N(\mc{S})$, $\mc{G^M}$ is partitioned into three subgraphs:  
\begin{enumerate}
\item $\mc{S}$, which does not contain any bus in $\bar{\mc{M}}$, i.e., $\mc{S}\subseteq\bar{\mc{M}}^c$. 
\item $N(\mc{S})$, which is the vertex cut set itself, and may contain buses in $\bar{\mc{M}}$. 
\item $\mc{N}\backslash N[\mc{S}]$, which contains (not necessarily exclusively) all the remaining buses in $\bar{\mc{M}}$ after removing the cut set.
\end{enumerate}

An illustrative example with a cut $N(\mc{S})$ of size 2 is depicted in Figure \ref{3spgen} in Section \ref{upsec}. %on page 15.
%We now make a technial generalization of a vertex cut of $\mc{G^M}$ by \emph{allowing the third subgraph $\mc{N}\backslash N[\mc{S}]$ be an empty set}. Note that, when $\mc{N}\backslash N[\mc{S}] = \emptyset$, $\bar{\mc{M}} \subseteq N(\mc{S})$, and $N(\mc{S})$ is not really a vertex cut in its original definition. 
%\end{enumerate} 
%This generalization is specifically made for finding the sparsest unobservable attack in the remainder of the paper. 
We note that there is a slight %bit of 
abuse of notation in $N(S)$: In general, a cut set does not necessarily consist of exactly all the neighboring nodes of $S$. Nonetheless, as will be shown in the remainder of the paper, we need only care about the \emph{minimum} cut set, %under various circumstances,  
which indeed consists of exactly all the neighboring nodes of $S$, namely, $N(S)$. 
Leveraging the above notation, we now introduce a key type of vertex cut on $\mc{G^M}$. 
\begin{define}[Vulnerable vertex cut]
A vulnerable vertex cut of a connected augmented graph $\mc{G^M}$ is a vertex cut $N(\mc{S})$ for which %Denote the three subgraphs partitioned by the cut by $S, N(S)$ and $\mc{N}\backslash N[S]$ as above, 
%the number of alterable buses in $S\cup N(S)$ is no less than $\vert N(S)\vert+1$. 
$\vert \mc{U}^c \cap N[\mc{S}]\vert \ge \vert N(\mc{S})\vert+1$. 
\end{define}
In other words, the number of \emph{alterable} buses in  $N[\mc{S}]$ is no less than the cut size plus one. %$\vert N(S)\vert+1$. 
The reason for calling such a vertex cut ``vulnerable'' will be made exact later in Section \ref{upsec}. The basic intuition is the following. In order to have $\Delta\bm{\theta}_{\mc{M}} = 0$ (unobservability), the key is to have the phase angle changes on the cut $N(\mc{S})$ be zero, with power injection changes (which can only happen on the alterable buses) restricted in $N[\mc{S}]$. As will be shown later, this can be achieved if a cut $N(\mc{S})$ is ``vulnerable'' as defined above. 
We note that it is possible that no vulnerable vertex cut exists (e.g., in the extreme case that all buses are unalterable). 

Accordingly, we define the following variation on the vertex connectivity. 
\begin{define}[Vulnerable vertex connectivity]
The vulnerable vertex connectivity of an augmented graph $\mc{G^M}$, denoted by $\bar{\kappa}(\mc{G^M})$, is the size of the minimum vulnerable vertex cut of $\mc{G^M}$. If no vulnerable vertex cut exists, $\bar{\kappa}(\mc{G^M})$ is defined to be $\infty$. 
\end{define}
%Clearly, $\bar{\kappa}(\mc{G^M}) \ge \kappa(\mc{G^M}) \ge \kappa(G)$. 
We note that the concepts of vulnerable vertex cut and vulnerable vertex connectivity do not apply to the original graph $\mc{G}$. We immediately have the following lemma:
\begin{lem} \label{kplem}
If a vulnerable vertex cut exists, then $\bar{\kappa}(\mc{G^M})\le M = |\mc{M}|$. 
\end{lem}
\begin{proof}
Suppose a vulnerable vertex cut exists, and $\bar{\kappa}(\mc{G^M})\ge M+1$. Denote the minimum vulnerable vertex cut by $N(\mc{S})$ (cf. Notation \ref{ntn1}). Now consider the set $\mc{M}$: it is a vertex cut of $\mc{G^M}$ that separates the dummy bus and %i) $\mc{M} \subseteq S^c$, ii) $\mc{M}$ disconnects the dummy bus from 
$\bar{\mc{M}}^c$. Because %i) $S\cup N(\mc{S})\subseteq\bar{M}^c\cup\mc{M}$, and ii) 
there are at least $\bar{\kappa}(\mc{G^M})+1 \ge M+2$ alterable buses in $N[\mc{S}]\subseteq N[\bar{M}^c]$, %and hence at least $M+2$ alterable buses in $\bar{M}^c\cup\mc{M}$, 
$\mc{M}$ is also a \emph{vulnerable vertex cut}. This contradicts the minimum vulnerable vertex cut having size at least $M+1$.
\end{proof}

\subsection{Main result}
We now state the following theorem that gives an explicit solution of the sparsest unobservable attack problem in terms of the vulnerable vertex connectivity $\bar{\kappa}(\mc{G^M})$. 

\begin{thm} \label{mainthm}
For a connected grid $\mc{G=\{N,L},\bm{w}\}$, assume that the line reactances $x_l ~(l\in\mc{L})$ are independent continuous random variables strictly bounded away from zero from below. Given any $\mc{M}$ and $\mc{U}$, %denoting the set of buses with PMUs, 
the minimum sparsity of unobservable attacks, i.e., the global optimum of \eqref{l0min_g2}, equals $\bar{\kappa}(\mc{G^M})+1$ with probability one.
\end{thm}

We note that finding the minimum vulnerable vertex connectivity of a graph is computationally efficient. %: one way of computing it is by first transforming the graph into a new one, and then computing its minimum \emph{edge} cut for which many polynomial time algorithms are available. 
For polynomial time algorithms we refer the readers to \cite{nodesplit} and \cite{vazirani1992suboptimal}. In particular, vertex cuts are enumerated \cite{vazirani1992suboptimal} starting from the minimum and with increasing sizes, until a minimum vulnerable vertex cut is identified. %In summary, Theorem \ref{mainthm} not only shows that the sparsest attack \eqref{l0min_g2} can be found in polynomial time with probability one, but also crystalizes the graph-theoretic interpretation of the optimal solutions. In particular, the constructive solution based on the minimum vertex cut of the augmented graph $\mc{G^M}$ is guaranteed to be optimal for solving \eqref{l0min_g2}. %in other words, a sparser unobservable attack does not exist.
%Next, we will show that this graph-theoretic insight leads to a natural characterization of the potential impacts of unobservable attacks.
We now prove Theorem \ref{mainthm} 
by upper and lower bounding the minimum sparsity of unobservable attacks in the following two subsections. 

\subsection{Upper bounding the minimum sparsity of unobservable attacks} \label{upsec}
%Similar to \eqref{l0min_g2} as an equivalent form of \eqref{l0min}, an equivalent form of \eqref{theQ} is as follows:
%\begin{align}
%    \eqref{theQ} ~\Leftrightarrow~ \min_{\Delta\bm{\theta}_{\mc{K}^c}\ne \bm{0}}~ \Vert\bm{B}_{\mc{N}\mc{K}^c}\Delta\bm{\theta}_{\mc{K}^c}\Vert_0. \label{theQ2}
%\end{align}
%To upper bound on the minimum sparsity of unobservable attacks, %we exploit the fact that solving \eqref{l0min_g2} is equivalent to finding the sparsest non-zero vector in the \emph{range space} of $\bm{B}_{\mc{N}\mc{M}^c}$. 
We show that \emph{any} vulnerable vertex cut $N(\mc{S})$ provides an upper bound on the optimum of \eqref{l0min_g2} as follows. 
% a feasible unobservable attack with sparsity $|N(\mc{S})|+1$. 
%The following theorem provides an upper bound on the optimum of \eqref{l0min_g2}:
\begin{thm} \label{upthm}
For a connected grid $\mc{G}$ and a set of PMUs $\mc{M}$, for any vulnerable vertex cut of $\mc{G^M}$ denoted by $N(\mc{S})$ (cf. Notation \ref{ntn1}), there exists an unobservable attack of sparsity no higher than $|N(\mc{S})|+1$.
\end{thm}
%{\small
\begin{proof}
A vulnerable vertex cut $N(\mc{S})$ partitions $\mc{G^M}$ into $\mc{S}$, $N(\mc{S})$ and $\mc{N}\backslash \mc{N[S]}$, with $\mc{S}\subseteq \mc{M}^c$. 
Similarly to the range space interpretation of the sparsest unobservable attack \eqref{l0min_g2}, it is sufficient to show that there exists a non-zero vector in the range space of $\bm{B}_{\mc{N}\mc{S}}$ such that i) it has a sparsity no higher than $|N(\mc{S})|+1$, and ii) non-zero power injections occur only at the alterable buses. 

%We show that there exists a 
%From Lemma \ref{lemL}, $\bm{B}_{\mc{N}\mc{M}^c}$ always has a full column rank. %we can always choose a number of $\rank(\bm{B}_{\mc{N}\mc{K}^c})$ of its $N-k$ columns to form a full column rank submatrix with the \emph{same} range space. Thus, it is sufficient to prove for the case of $\rank(\bm{B}_{\mc{N}\mc{K}^c}) = |\mc{K}^c| = N-k$.
By re-indexing the buses, WLOG, i) let $\mc{S} = \{1,2,\ldots,|\mc{S}|\}$, and ii) let $\bm{B}_{\mc{N}\mc{S}}$ have the following partition as depicted in Figure \ref{multicol}:
\begin{enumerate}
\item The top submatrix $\bm{B}_{\mc{S}\mc{S}}$ is an $|\mc{S}|\times |\mc{S}|$ %full-rank 
matrix. %(cf. Lemma \ref{lemL}).
\item The middle submatrix (which will be shown to be $\bm{B}_{N(\mc{S})\mc{S}}$) consists of all the remaining rows, each of which has at least one \emph{non-zero} entry.
\item The bottom submatrix is an \emph{all-zero} matrix.
\end{enumerate}
In particular, from the definition of the Laplacian, %\eqref{lapgr},
the middle submatrix of $\bm{B}_{\mc{N}\mc{S}}$, as described above, %that consists of all the rows, other than the first $N-M$ rows, that each has at least one \emph{non-zero} entry
is exactly $\bm{B}_{N(\mc{S})\mc{S}}$ %for the following reasons: By
because its row indices correspond to those %for any column of $\bm{B}_{\mc{N}\mc{M}^c}$, $\bm{b}_i, (i=1,\ldots,N-M,)$, its non-zero entries correspond to bus $i$ and those
buses not in $\mc{S}$ but connected to at least one bus in $\mc{S}$. %$bus $i$. %Thus, the set of row indices of that con (cf. Figure \ref{multicol}), denoted by $\mc{C}$, is the set of the buses that are connected to at least one of the buses in $\mc{M}^c$,
%Note that %the middle and
%the bottom sub-matrix can be degenerate. 

From the definition of the vulnerable vertex cut, $\vert \mc{U}^c \cap N[\mc{S}]\vert \ge \vert N(\mc{S})\vert+1$. Now, pick any set of $\vert N(\mc{S})\vert+1$ alterable buses in $\mc{U}^c \cap N[\mc{S}]$, denote this set by $\mc{A}$, and denote the other buses in $N[\mc{S}]$ by $\tilde{U} \triangleq N[\mc{S}]\backslash \mc{A}$. Clearly, $|\tilde{U}| = |S|-1$. 
%Thus, $\vert \mc{U} \cap\left(\mc{S}\cup N(\mc{S})\right)\vert \le \vert \mc{S}\vert -1$. In other words, the number of unalterable buses in $\mc{S}\cup N(\mc{S})$ is no greater than 
%$\vert\mc{S}\vert -1$. Denote these unalterable buses by $\tilde{\mc{U}}\triangleq \mc{U} \cap\left(\mc{S}\cup N(\mc{S})\right)$. 
Therefore, $\bm{B}_{\tilde{\mc{U}}\mc{S}}$ (which is a submatrix of $\bm{B}_{N[\mc{S}]\mc{S}}$) has $|S|$ columns but only $|S|-1$ rows, and is hence column rank deficient. 

Now, we let $\Delta\bm{\theta}_{\mc{S}}$ be a non-zero vector in the null space of $\bm{B}_{\tilde{\mc{U}}\mc{S}}$:
\begin{align}
\bm{B}_{\tilde{\mc{U}}\mc{S}}\Delta\bm{\theta}_{\mc{S}} = \bm{0}. \label{thetaf}
\end{align}
Then, we construct an attack vector $\Delta\bm{P} =  \bm{B}_{\mc{N}\mc{S}} \Delta\bm{\theta}_{\mc{S}}$: it has some possibly non-zero values at the indices that correspond to $\mc{A}$, and has \emph{zero values at all other indices.} Thus,
\begin{align}
\Vert\Delta\bm{P}\Vert_0\le |\mc{A}| = |N(\mc{S})|+1. \label{pf_0}
\end{align}
\end{proof}
%}

\begin{figure}[tb!]
  \centering
  \subfigure[Block representation of $\bm{B}_{\mc{N}\mc{S}}$.]{
  \includegraphics[scale = 0.55]{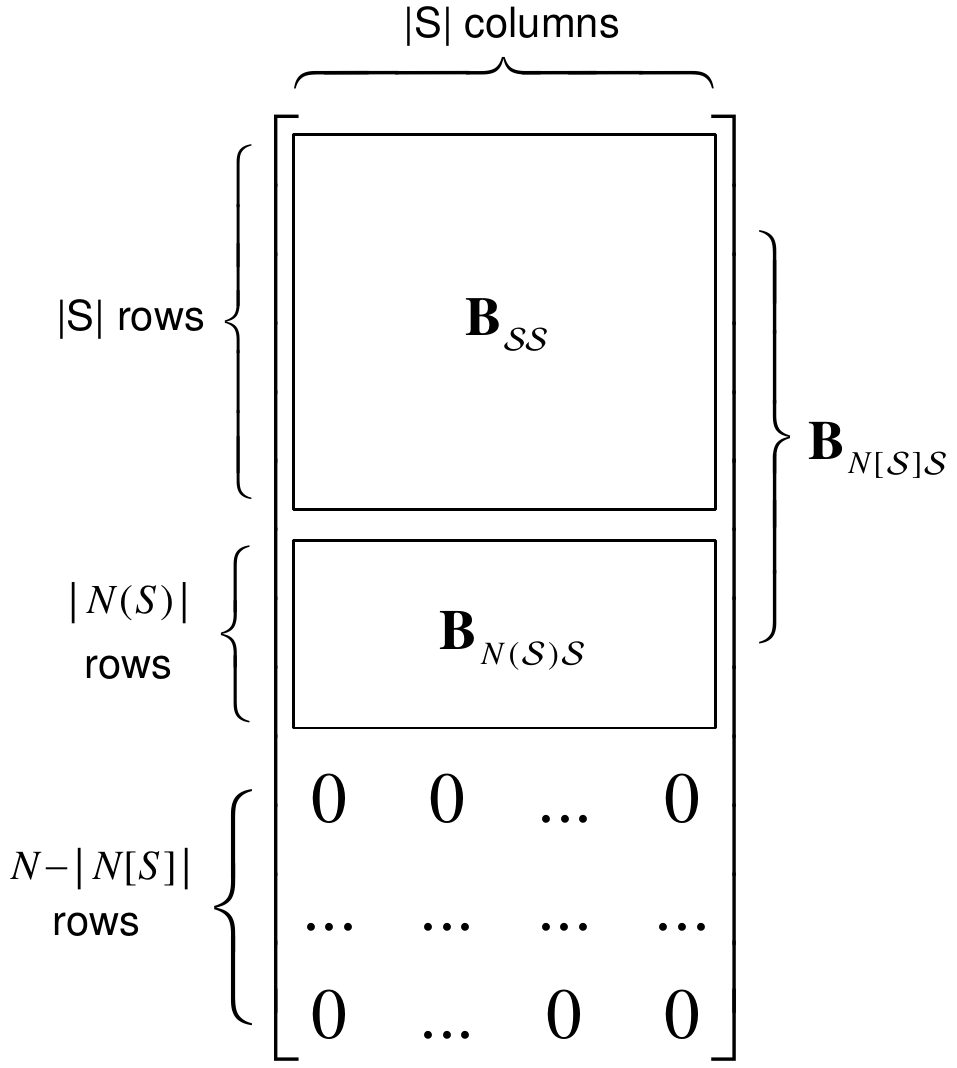}
  \label{multicol}
  }
  \subfigure[A 3-sparse unobservable power injection attack.]{
  \includegraphics[scale = 0.65]{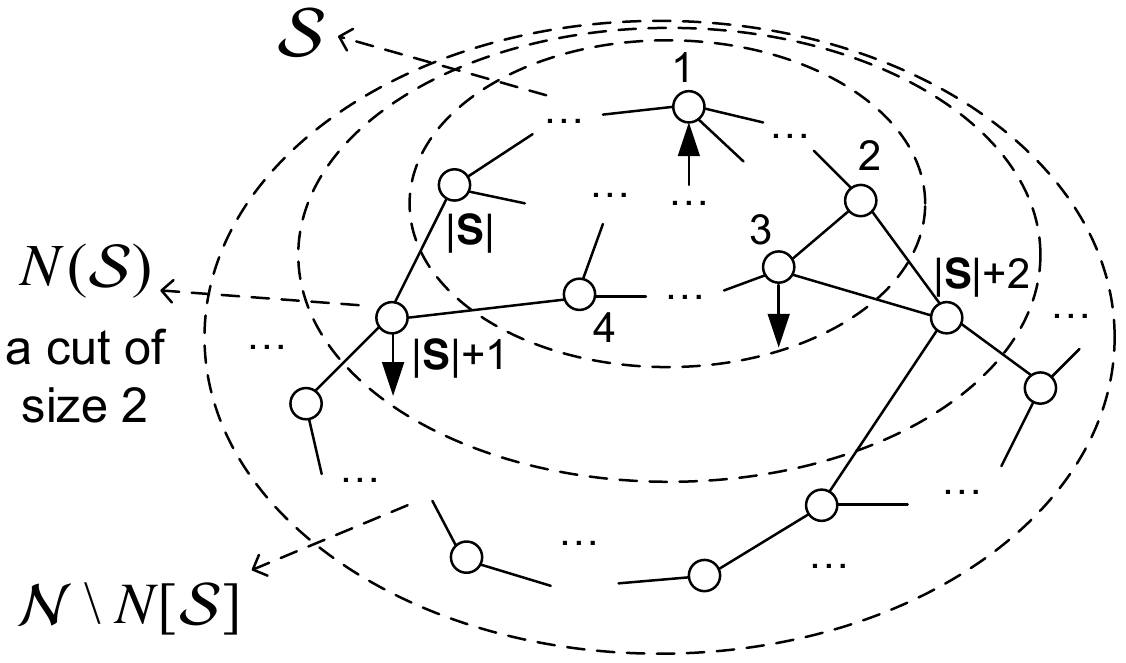}
  \label{3spgen}
  }
  \caption{Sparse attacks with voltage phase angle changes restricted to buses $1,2,\ldots |\mc{S}|$.}
  \label{genex}
\end{figure}

Theorem \ref{upthm} explains our terminology of a ``vulnerable vertex cut'', since if a vertex cut is vulnerable, it leads to an unobservable attack. %Moreover, Theorem \ref{upthm} provides a \emph{constructive solution} for unobservable attacks. 
 If a vulnerable vertex cut of $\mc{G^M}$ exists, applying Theorem \ref{upthm} to the \emph{minimum} one, we have that the optimum of \eqref{l0min_g2} is upper bounded by $\bar{\kappa}(\mc{G^M})+1$. If no vulnerable vertex cut exists, $\bar{\kappa}(\mc{G^M})+1 = \infty$ is a trivial upper bound. 

We now provide a graph-theoretic interpretation of Theorem \ref{upthm}. As shown in Figure \ref{multicol} and \ref{3spgen}, all the buses can be partitioned into three subsets $\mc{S}, N(\mc{S})$ and $\mc{N}\backslash N[\mc{S}]$, corresponding to the row indices of the top, middle and bottom submatrices of $\bm{B}_{\mc{N}\mc{S}}$, respectively. $N(\mc{S})$ is a vulnerable vertex cut of $\mc{G^M}$ that separates $\mc{S}$ from $\mc{N}\backslash N[\mc{S}]$. The %$|N(\mc{S})|+1$-
sparse attack $\Delta\bm{P}$ (cf. \eqref{pf_0}) is formed by injecting/extracting power at $|N(\mc{S})|+1$ alterable buses in $N[\mc{S}]$, such that the phase angle changes at $\mc{N}\backslash \mc{S}$ are all zero. Note that $(\mc{N}\backslash \mc{S})\supseteq \mc{M}$. %Note that the (possibly non-zero) phase angle changes at $\mc{M}^c$ are given by $\Delta\bm{\theta}$ \eqref{thetaf}. %Finally, we note that $\mc{S}, N(\mc{S}),\mc{N}\backslash\mc{S}\backslash N(\mc{S})$ correspond to the row indices of $\bm{F},\bm{A},$ and the all zero submatrix as in Figure \ref{multicol} respectively.
The example with $|N(\mc{S})| = 2$ in Figure \ref{3spgen} illustrates a $3$-sparse attack with power injection/extractions at (assumed alterable) buses $1,3$ and $|\mc{S}|+1$, such that the phase angle changes at $\mc{N}\backslash \mc{S}$ are all zero. %We name the $N(\mc{S})$ the boundary set of ,
%In this example, the set $\mc{M}^c = \{1,2,\ldots,N-M\}$ interacts with the set $\mc{M}\backslash N(\mc{M}^c)$ via only \emph{two} buses --- $N(\mc{M}^c) = \{N-M+1, N-M+2\}$. As a result, there exists a 3-sparse attack that injects/extracts power at bus 1, $N-M+1$ and $N-M+2$, such that the phase angle changes at $\mc{M}$ are all zero. %In other words, there is no power flowing outside the subgraph induced by $N[\mc{M}^c]$, and none of the states outside $\mc{S}$ is affected by this attack.
%This graph-theoretic interpretation provides a constructive solution of the optimal solutions in Theorem \ref{mainthm} with a constructive solution based on the minimum vulnerable vertex cut of $\mc{G^M}$. %is guaranteed to be optimal for solving \eqref{l0min_g2}.

We end this subsection by introducing a notion of ``potential impact'' of unobservable attacks. We make the following observation: As long as an attacker takes control of all the power injections in a vulnerable vertex cut $N(\mc{S})$ (assuming they are alterable), it can always \emph{cancel out the effects of anything that happens within $N[\mc{S}]$} on the measurements taken in $\mc{M}~(\subseteq \mc{N}\backslash\mc{S})$. 
Thus, by taking control of all the buses in $N(\mc{S})$, an attacker can successfully \emph{hide} from the system operator a power injection attack with a zero norm as large as %$\Delta\bm{P}$ with $|\Delta\bm{P}|$
\begin{align}
|N[\mc{S}]| = |N(\mc{S})| + |\mc{S}|~(\gg |N(\mc{S})| + 1).
\end{align}

Accordingly, we introduce the following definition. 

\begin{define} \label{impactdef}
The potential impact of unobservable attacks associated with a vulnerable vertex cut $N(\mc{S})$ is defined as $|N[\mc{S}]|$.
\end{define}
\begin{RK}
Definition \ref{impactdef} is one characterization of attack impact based solely on graph theoretic properties. In practice, there %can surely be 
are many different notions of attack impact depending on, e.g., the interpretation of the attacks and the operating objective of the system. 
\end{RK}

Employing Definition \ref{impactdef}, %notion of potential impacts of unobservable attacks,
we can \emph{differentiate} the potential impacts of multiple sparsest unobservable attacks \emph{with the same sparsity}. %In particular, multiple minimum vertex cuts can exist for the same augmented graph $\mc{G^M}$. Then, each of these cuts leads to a different sparsest attack of the same size (constructed by controlling the buses in this cut as well as one other bus disconnected from $\mc{M}$ by it). However, different cuts may disconnect different portions of the network from $\mc{M}$, leading to vastly different potential impacts of unobservable attacks.
An illustration is depicted in Figure \ref{impfig}. In this example, two vulnerable vertex cuts both of size two, $N(\mc{S}_1) = \{V_{1A}, V_{1B}\}$ and $N(\mc{S}_2) = \{V_{2A}, V_{2B}\}$, are enclosed by solid ovals. Accordingly, both cuts enable 3-sparse unobservable attacks. However, their potential impacts are significantly different. Cut $N(\mc{S}_2)$ only disconnects one other bus, namely $\mc{S}_2 = \{V_{2C}\}$ from the set of PMUs $\mc{M}$, and hence its potential impact equals $|N[\mc{S}_2]| = 3$. In comparison, cut $N(\mc{S}_1)$ disconnects all the vertices above $N(\mc{S}_1)$ from $\mc{M}$, and hence its potential impact equals $|N[\mc{S}_1]| \gg 3$. With this definition of potential impact, it is then natural for an attacker to \emph{seek the sparsest unobservable attack with the greatest potential impact}.

As an immediate byproduct of the analysis of potential impact, by letting $\mc{S} = \mc{M}^c$, we obtain the \emph{maximum} potential impact of all unobservable attacks in a power network:%, as in the following corollary:
\begin{cor}\label{maxcor}
For a connected power grid $\mc{G=\{N,L},\bm{w}\}$, given any $\mc{M}$ denoting the PMU locations, the maximum potential impact among all the unobservable attacks equals $|N[\mc{M}^c]|$.
\end{cor}

\begin{figure}[tb]
  \centering
  \includegraphics[scale = 0.85]{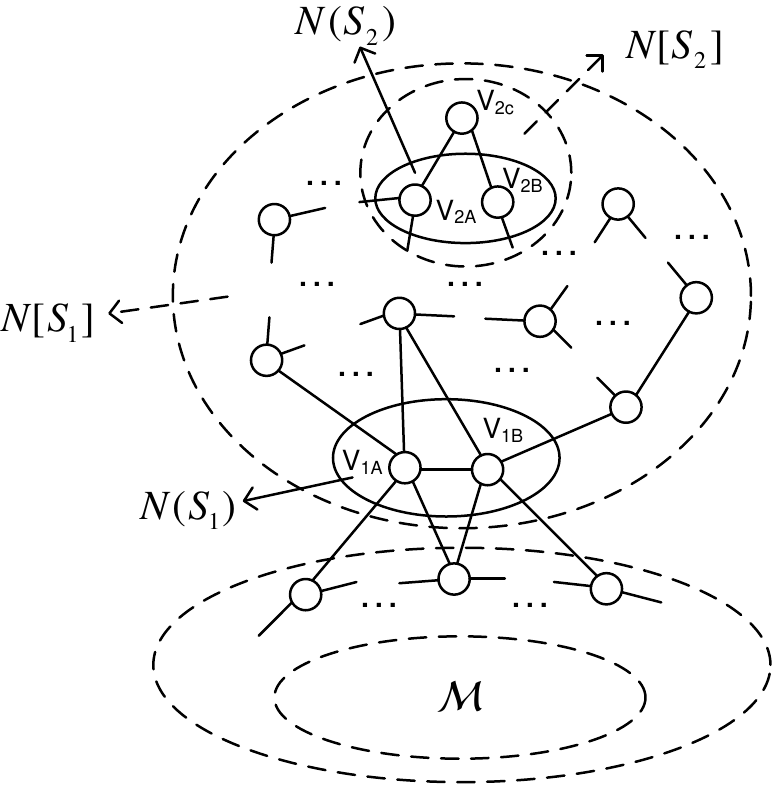}
  \caption{An illustration of two vulnerable vertex cuts with the same size but different potential impacts.}
  \label{impfig}
\end{figure}

\subsection{Lower bounding the sparsity of unobservable attacks}
We first define the following property of a matrix $\bm{H}\in\mbb{R}^{N\times N}$, which will be shown to be equivalent to having $\spr(\bm{H}) = N$.
\begin{ppt}[An equivalent condition for having a full structural rank] \label{cond1def}
\begin{align}
&\forall n = 1,2,\ldots, N, \text{ and for any } n\times N \text{ submatrix of }\bm{H}, \nn\\
&\text{the submatrix has at least $n$ columns each with at least} \nn\\ 
&\text{one non-zero entry.} \nn%\label{cond1}
\end{align}
\end{ppt}

We have the following lemma whose proof is relegated to Appendix \ref{secprfcond1}:
\begin{lem} \label{cond1equ}
Property \ref{cond1def} %(cf. Definition \ref{cond1def})
is equivalent to having $\spr(\bm{H}) = N$.
\end{lem}

%We now state the following theorem in lower bounding the sparsity of unobservable attacks. %will show that Property \ref{cond1def} (and hence a full structural rank) relates to the vertex connectivity of the grid in the next theorem: %that addresses \eqref{l0min_g2}: %We now provide the following theorem addressing \eqref{l0min_g2}:
%\begin{thm} \label{minvcthm}
%For a connected power grid $\mc{G=\{N,L},\bm{w}\}$, assume that the line reactances $x_l ~(l\in\mc{L})$ are independent continuous random variables strictly bounded away from zero from below. 
%For \emph{any} set of buses $\mc{M}\subset\mc{N}$ with $M \le\bar{\kappa}(\mc{G^M})$ where $M \triangleq |\mc{M}|$, in order to have $\Delta\bm{\theta}_{\mc{M}} = \bm{0}$, one of the following must be true with probability one:
%\begin{itemize}
%\item There is no power injection attack at any bus in the grid, i.e., $\Delta\bm{P}=\bm{0}$; or
%\item There must be at least $M+1$ buses with non-zero power injections from an attack, i.e., $\Vert\Delta\bm{P}\Vert_0\ge M+1$.
%\end{itemize}
%\end{thm}

We now prove the lower bounding part of Theorem \ref{mainthm}, namely, with probability one, all unobservable power injection attacks $\Delta \bm{P}$ must have $\Vert\Delta\bm{P}\Vert_0\ge \bar{\kappa}(\mc{G^M}) + 1$. 
The key idea %behind proving Theorem \ref{minvcthm} 
is in showing that the equivalence between Property \ref{cond1def} and a full structural rank (cf. Lemma \ref{cond1equ}) implies a connection between the vulnerable vertex connectivity and the feasibility condition of unobservable attacks (cf. Lemma \ref{lemfeas1}).

%{\small
\begin{proof}[Proof of $\Vert\Delta\bm{P}\Vert_0\ge \bar{\kappa}(\mc{G^M}) + 1$ for unobservable $\Delta \bm{P}$, w.p.1] %Theorem \ref{minvcthm}
%From Lemma \ref{lemaug}, 
We focus on $\mc{G^M}$ and consider its corresponding Laplacian $\bm{B}$.  
%We prove the case of $M = \kappa(\mc{G})$. The case of $M < \kappa(\mc{G})$ follows directly. For notational simplicity, we denote $\kappa(\mc{G})$ by $\kappa$.
Suppose there exists a power injection attack $\Delta\bm{P}\ne\bm{0}$ such that
\begin{align}
\Delta\bm{\theta}_\mc{M} = \bm{0} %\text{ where } \bm{P} = \bm{B}\bm{\theta}, \label{proofasump1}\\
\text{ and }\Vert\Delta\bm{P}\Vert_0\le{\bar{\kappa}(\mc{G^M})}. \label{proofasump2}
\end{align}

%From Lemma \ref{kplem}, $\bar{\kappa}(\mc{G^M})\le M$, 
%For any set of buses $\mc{M}\in\mc{N}$ with $M = |\mc{M}| \le \kappa(\mc{G^M})$, 

Denote the buses with non-zero power injection changes by $\mc{A}\subseteq\mc{U}^c$, and hence $\Delta\bm{P}_{\mc{A}^c} = \bm{0}$. From \eqref{proofasump2}, $|\mc{A}|\le \bar{\kappa}(\mc{G^M}), \Delta\bm{\theta}_{\mc{M}^c}\!\ne\!\bm{0}$, and $\bm{0} \!=\! \Delta\bm{P}_{\mc{A}^c} \!=\! \bm{B}_{\mc{A}^c\mc{M}^c}\Delta\bm{\theta}_{\mc{M}^c}$, %, where $N = |\mc{N}|$. From $\bm{P} = \bm{B\theta}$, we have $\bm{0} = \bm{P}_\mc{Z} = \bm{B}_\mc{Z}\bm{\theta}$, where $\bm{B}_\mc{Z}\in\mbb{R}^{Z\times N}$ is the submatrix of $\bm{B}$ by selecting the rows with indices in $\mc{Z}$. Moreover, from \eqref{proofasump1},
%\begin{align}
%\bm{0} = \Delta\bm{P}_{\mc{A}^c} = \bm{B}_{\mc{A}^c\mc{M}^c}\Delta\bm{\theta}_{\mc{M}^c}. \label{thm1med}
%\end{align}
%where $\bm{B}_\mc{Z, N\backslash K}\in\mbb{R}^{Z\times (N-\kappa)}$ is the submatrix of $\bm{B}$ by selecting the entries with row indices in $\mc{Z}$ and column indices in $\mc{N\backslash K}$.
%\eqref{thm1med}
%This implies
implying that $\bm{B}_{\mc{A}^c\mc{M}^c}$ is column rank deficient. We first consider the case that a vulnerable vertex cut exists, i.e., $\bar{\kappa}(\mc{G^M}) < \infty$. The proof for the case of $\bar{\kappa}(\mc{G^M}) = \infty$ follows similarly. 
For notational simplicity, we will use $\bar{\kappa}$ instead of $\bar{\kappa}(\mc{G^M})$ in the remainder of the proof.  
\vspace{5pt}

\paragraph{If a vulnerable vertex cut exists, i.e., $\bar{\kappa} < \infty$}~

We will prove that, for all $\mc{A}\subseteq \mc{U}^c$ with $|\mc{A}|\le \bar{\kappa}$, \emph{$\bm{B}_{\mc{A}^c\mc{M}^c}$ is of full column rank with probability one}, i.e., \eqref{proofasump2} can only happen with probability zero. From Lemma \ref{kplem}, $\bar{\kappa}\le M$. %For proving that $\bm{B}_{\mc{A}^c\mc{M}^c}$ is of full column rank, 
It is then sufficient to prove for the ``worst cases'' %of all such $\mc{A}$ 
with $|\mc{A}| = \bar{\kappa}=M$, i.e., $|\mc{A}^c| = |\mc{M}^c| = N-\bar{\kappa}$ and $\bm{B}_{\mc{A}^c\mc{M}^c}$ is a square matrix. %It is hence sufficient to show that, \emph{all $(N-M)\times(N-M)$ submatrices of $\bm{B}$ are 
From Theorem \ref{feasthm} and Lemma \ref{cond1equ}, it is sufficient to show that %all $(N-{M})\times(N-{M})$ submatrices of $\bm{B}$
%\emph{for all such $\mc{A}$ with $|\mc{A}|=M$}, 
\emph{$\bm{B}_{\mc{A}^c\mc{M}^c}$ %\left(\in \mbb{R}^{(N-M)\times(N-M)}\right)$ 
satisfies Property \ref{cond1def}}, and hence is of full rank with probability one.
Recall from the definition of the Laplacian $\bm{B}$ %\eqref{lapgr}
that, for any column (or row) of $\bm{B}$, $\bm{b}_i, (i=1,\ldots,N)$, its non-zero entries correspond to bus $i$ and those buses that are connected to bus $i$. With this, we now prove that $\bm{B}_{\mc{A}^c\mc{M}^c}$ satisfies Property \ref{cond1def}. 

Consider any set of $n$ ($\le N-\bar{\kappa}$) buses in $\mc{A}^c$, denoted by $\tilde{\mc{N}}$. 

i) If $\tilde{\mc{N}}\subseteq \mc{M}^c$: 
Based on the definition of the Laplacian $\bm{B}$, the $n$ columns of $\bm{B}_{\tilde{\mc{N}}\mc{M}^c}$ \emph{that correspond to the buses $\tilde{\mc{N}}$ themselves} each has at least one non-zero entry. %Thus, $\bm{B}_{\tilde{\mc{N}}\mc{M}^c}$ has at least $n$ columns each of which has at least one non-zero entry. 

ii) If $\tilde{\mc{N}}\cap \mc{M} \ne \emptyset$:  
We prove that $N(\tilde{\mc{N}})$ must contain at least $\bar{\kappa}$ buses. This is because, otherwise, $|N(\tilde{\mc{N}})|\le \bar{\kappa}-1$, contradicting that $\bar{\kappa}$ is the minimum size of vulnerable vertex cuts %the vertex connectivity of the grid, namely 
%$\bar{\kappa}(\mc{G^M})$ would be strictly less than ${M}$ 
for the following reasons: %(with the neighborhood of %neighborhood of 
%$\tilde{\mc{N}}$ as the minimum vertex cut), 
\begin{enumerate}
\item $\mc{A}\subseteq\tilde{\mc{N}}^c$, and thus $\tilde{\mc{N}}^c$ has at least $|\mc{A}| = \bar{\kappa}$ alterable buses. 
\item $|N(\tilde{\mc{N}})|\le \bar{\kappa}-1$ implies that $\tilde{\mc{N}}^c\backslash N(\tilde{\mc{N}})\ne\emptyset$, and thus $N(\tilde{\mc{N}})$ is a vertex cut that separates $\tilde{\mc{N}}$ and $\tilde{\mc{N}}^c\backslash N(\tilde{\mc{N}})$. 
\item Because $\tilde{\mc{N}}\cap \mc{M} \ne \emptyset$ and $\mc{M}$ are pairwise connected in $\mc{G^M}$, $\mc{M}\subseteq N[\tilde{\mc{N}}]$. Thus, $\tilde{\mc{N}}^c\backslash N(\tilde{\mc{N}})$ and $\mc{M}$ are disjoint.  
\end{enumerate}
From 1), 3), and the fact that $|N(\tilde{\mc{N}})|\le \bar{\kappa}-1$, we observe that $N(\tilde{\mc{N}})$ is a \emph{vulnerable vertex cut} of size $\bar{\kappa}-1$, contradicting $\bar{\kappa}$ being the vulnerable vertex connectivity. %our assumption. %consider the $n\times N$ submatrix of $\bm{B}_{\tilde{\mc{N}}\mc{N}}$, %denoted by $\bm{B}'$, whose rows are selected according to $\tilde{\mc{N}}$.

Now, based on the definition of the Laplacian $\bm{B}$, %\eqref{lapgr}
the $n\times N$ submatrix \emph{$\bm{B}_{\tilde{\mc{N}}\mc{N}}$ must have at least $n + \bar{\kappa}$ columns each of which has at least one non-zero entry for the following reasons:}
\begin{itemize}
\item The $n$ columns of $\bm{B}_{\tilde{\mc{N}}\mc{N}}$ that correspond to the buses $\tilde{\mc{N}}$ themselves each has at least one non-zero entry.
\item As $\tilde{\mc{N}}$ are connected to at least $\bar{\kappa}$ other buses, each one of these $\bar{\kappa}$ neighbors of $\tilde{\mc{N}}$ corresponds to one column of $\bm{B}_{\tilde{\mc{N}}\mc{N}}$ that has at least one non-zero entry.
\end{itemize}
Accordingly, the $n\times (N-{\bar{\kappa}})$ submatrix $\bm{B}_{\tilde{\mc{N}}\mc{M}^c}$ has at least $n$ columns each of which has at least one non-zero entry. 

Summarizing i) and ii), $\bm{B}_{\mc{A}^c\mc{M}^c}$ satisfies Property \ref{cond1def}, and is thus of full column rank with probability one. Therefore, \eqref{proofasump2} can only happen with probability zero. %and from \eqref{thm1med} $\bm{\theta}_\mc{N\backslash K} = \bm{0} \Rightarrow \bm{\theta} = \bm{0} \Rightarrow \bm{P} = \bm{0}$, contradicting the assumption that $\bm{P} \ne \bm{0}$.
\vspace{5pt}

\paragraph{If no vulnerable vertex cut exists, i.e., $\bar{\kappa} = \infty$}

If $\mc{M}=\mc{N}$, i.e., all buses have PMUs, then clearly no unobservable attack exists. We now focus on $M\le N-1$. Suppose $|\mc{A}|\ge M+1$. Consider the set $\bar{M}$ containing $\mc{M}$ and the dummy bus. $\Delta\bm{\theta}_\mc{M} = \bm{0}$ (cf. \eqref{proofasump2}) implies that $\mc{A} \subseteq N[\bar{\mc{M}}^c]$, and thus $N[\bar{\mc{M}}^c]$ has at least $|\mc{A}| \ge M+1$ alterable buses. Since $\mc{M} ~(= N(\bar{\mc{M}}^c))$ separates the dummy node and $\mc{N}\backslash \mc{M}$, $\mc{M}$ is a \emph{vulnerable vertex cut}. This contradicts the nonexistence of a vulnerable vertex cut. 
Therefore, $|\mc{A}|\le M$. In this case, the same proof as in the above case i) when a vulnerable vertex cut exists applies, and \eqref{proofasump2} can only happen with probability zero. 
\end{proof}
%}
%
%Theorem \ref{minvcthm} provides a lower bound on the optimum of \eqref{l0min_g2} which holds with probability one: no matter which set of $M$ buses' phase angles are monitored by the system operator, %must have their voltage phase angles unchanged,
%\begin{itemize}
%\item if $M\le\bar{\kappa}(\mc{G^M})$, there must be at least $M+1$ power injections in any non-zero unobservable attack;
%\item if $M>\bar{\kappa}(\mc{G^M})$, there must be at least $\bar{\kappa}(\mc{G^M})+1$ power injections in any non-zero unobservable attack.
%\end{itemize}
%The second point above holds by applying Theorem \ref{minvcthm} to a subset of $\bar{\kappa}(\mc{G^M})$ PMUs among a total of $M~(>\bar{\kappa}(\mc{G^M}))$ PMUs. Note that any attack that is unobservable by the $M$ PMUs is always unobservable by a subset of them.

% is no less than $k+1$  as long as $k\le\kappa(\mc{G})$. %In particular, to zeroize the voltage phase angles at any $\kappa(\mc{G})$ buses, $\kappa(\mc{G})+1$ is a lower bound on the number of buses with non-zero power injections in the entire grid.
%On the other hand, $k+1$ non-zero power injections are always sufficient to zeroize the voltage phase angles at any subset of $k$ buses, as stated in the following lemma whose proof is relegated to the Appendix:

With the proofs of upper and lower bounds, we have now proved Theorem \ref{mainthm}. In addition, from the proof of Theorem \ref{upthm}, we have a \emph{constructive solution} of the sparsest unobservable attack in polynomial time. We conclude this section by noting the following fact similar to that in Section \ref{secfeas}: the minimum sparsity %and the potential impacts 
of unobservable attacks is fully determined with probability one by the \emph{network topology, the locations of the alterable buses, and the locations of the PMUs}.

\begin{table*}[t]
\normalsize
\caption{Algorithm 1} \label{alg1table}
\centering
\begin{tabular}[c]{@{} l @{} l @{} p{8.0cm} @{}}
\multicolumn{3}{@{}c}{~Greedy algorithm for PMU placement for countering power injection attacks }\\
%\multicolumn{3}{@{}c}{~satisfy $P_e$ for line networks}\\
\hline
\multicolumn{3}{@{}l}{~~Place the $1^{st}$ PMU at bus 1.}\\ %Initialize $M=1$, and
%\multicolumn{3}{@{}l}{}\\
\multicolumn{3}{@{}l}{~~Repeat}\\
~~~~~&If no unobservable attack exists given the current set of PMUs $\mc{M}$, stop. \\
%~~~~~&$M\leftarrow M+1$.\\
~~~~~&Step 1: Find all the minimum vulnerable vertex cuts of $\mc{G}^{N[\mc{M}]}$; \\ %the augmented graph 
~~~~~&~~~~~~~~ among them, find the cut with the greatest potential impact, \\
~~~~~&~~~~~~~~ denoted by $C(\mc{G}^{N[\mc{M}]})$.\\
~~~~~&Step 2: Among all the buses disconnected from $N[\mc{M}]$ by $C(\mc{G}^{N[\mc{M}]})$ \\
~~~~~&~~~~~~~~ as well as those in the cut set $C(\mc{G}^{N[\mc{M}]})$, place the next PMU \\
~~~~~&~~~~~~~~ at the one such that the resulting maximum potential impact \\
~~~~~&~~~~~~~~ among all the remaining unobservable attacks is minimized. \\
\hline
\end{tabular}
\end{table*}

\section{Numerical evaluation} \label{secnum} %  of sparsest unobservable attacks and their potential impacts
In this section, we evaluate the sparsest unobservable attacks and their potential impacts when the system operator deploys PMUs at optimized locations. We first provide an efficient algorithm for optimizing PMU placement by the system operator. Next, we provide comprehensive evaluation of our analysis and algorithms in multiple IEEE power system test cases as well as large-scale Polish power systems. Our MATLAB codes are openly available for download\footnote{The codes can be found at \url{http://www.princeton.edu/~yuez/pubs.html}}. %multiple standard test systems and

\subsection{Optimization of PMU placement for attack detection}
We have seen in Section \ref{secmin} that the minimum sparsity and potential impacts of unobservable attacks are determined fully by the network topology, the locations of the alterable buses, and the PMU placement. Note that, unlike network states and parameters which can vary over short and medium time scales, the transmission network topology %(or the set of possible network topologies) 
and the alterable buses typically stay the same over relatively long time scales. This motivates the system operator to optimize the PMU placement according to this information. %as transmission network expansion often takes years to happen.
%On the other hand, PMU installation in transmission networks is typically costly. Thus, it is preferable for the system operator to deploy PMUs in a greedy manner, as new PMUs can be added to the existing set whenever additional installation becomes cost-effective.

For the best performance in countering power injection attacks, the system operator wants to \emph{raise the minimum sparsity of unobservable attacks, as well as mitigate the maximum potential impact of unobservable attacks}. Algorithm 1 (cf. Table \ref{alg1table}) is developed for the system operator to greedily place PMUs to pursue both objectives. In this algorithm, we have assumed that the \emph{second PMU model} in Section \ref{secsensemod} is employed, and the algorithm can be adapted to the first PMU model by replacing $N[\mc{M}]$ with $\mc{M}$.

%\begin{table*}[t]
%\normalsize
%\caption{Algorithm 1} \label{alg1table}
%\centering
%\begin{tabular}[c]{@{} l @{} l @{} p{8cm} @{}}
%\multicolumn{3}{@{}c}{~Greedy algorithm for PMU placement for countering power injection attacks }\\
%%\multicolumn{3}{@{}c}{~satisfy $P_e$ for line networks}\\
%\hline
%\multicolumn{3}{@{}l}{~~Place the $1^{st}$ PMU at bus 1.}\\ %Initialize $M=1$, and
%%\multicolumn{3}{@{}l}{}\\
%\multicolumn{3}{@{}l}{~~Repeat}\\
%~~~~~&If no unobservable attack exists given the current set of PMUs $\mc{M}$, stop. \\
%%~~~~~&$M\leftarrow M+1$.\\
%~~~~~&Step 1: Find all the minimum vertex cuts of the augmented graph $\mc{G}^{N[\mc{M}]}$; \\
%~~~~~&~~~~~~~~ among them, find the one with the greatest potential impact, \\
%~~~~~&~~~~~~~~ denoted by $C(\mc{G}^{N[\mc{M}]})$.\\
%~~~~~&Step 2: Among all the buses disconnected from $N[\mc{M}]$ by $C(\mc{G}^{N[\mc{M}]})$ \\
%~~~~~&~~~~~~~~ as well as those in the cut set $C(\mc{G}^{N[\mc{M}]})$, place the next PMU \\
%~~~~~&~~~~~~~~ at the one such that the resulting maximum potential impact \\
%~~~~~&~~~~~~~~ among all the remaining unobservable attacks is minimized. \\
%\hline
%\end{tabular}
%\end{table*}

Algorithm 1 is essentially a successive cut/attack elimination procedure. The purpose of Step 1 is to identify the sparsest unobservable attack with the greatest potential impact. 
Specifically, Step 1 can be performed as follows: %using the following procedure:
\begin{enumerate}
\item Assign arbitrarily one of the buses in $\mc{M}$ as the \emph{source} node;
\item For each of the buses in $\mc{N}\backslash N[\mc{M}]$, assign it as the \emph{destination} node, and compute all the minimum vulnerable vertex cuts that separate such a source-destination pair. 
\item Among all the computed source-destination vertex cuts that have the same minimum size, compute their corresponding potential impacts, and select the minimum vertex cut with the greatest potential impact, denoted by $C(\mc{G}^{N[\mc{M}]})$.
\end{enumerate}
We note that all the minimum vulnerable vertex cuts can be enumerated in polynomial time (c.f. \cite{vazirani1992suboptimal}). 
%Since the computational complexity of finding a minimum source-destination vertex cut is no greater than $O(N^3)$, the overall computational complexity of Step 1 is no greater than $O(N^4)$. 
In our numerical evaluation using MATLAB on a laptop with Intel Core i7 3.1-GHz CPU and 8 GB of RAM, it takes less than $0.2$ seconds on average for every PMU placed for the IEEE 300 bus systems. This per-PMU time increases to about 50 seconds for the Polish 3012 bus system. 
In Step 2, our primary goal is to ensure that the cut set $C(\mc{G}^{N[\mc{M}]})$ found in Step 1 \emph{does not remain a legitimate vertex cut} after placing the next PMU. This can be achieved by placing the next PMU among the buses disconnected from $N[\mc{M}]$ by $C(\mc{G}^{N[\mc{M}]})$ as well as those in $C(\mc{G}^{N[\mc{M}]})$. Among such candidate buses, we choose the one that renders the \emph{minimum} maximum potential impact among all the remaining unobservable attacks (cf. Corollary \ref{maxcor}) had the next PMU been placed at it.

%\textcolor{red}{not optimal}.

\subsection{Numerical evaluation of unobservable attacks vs. number of PMUs}
We %now provide numerical evaluation results
evaluate our results in the IEEE 30-bus, IEEE 57-bus, IEEE 118-bus, IEEE 300-bus, Polish 2383-bus, Polish 2737-bus, and Polish 3012-bus systems. The evaluation is performed based on the software toolbox MATPOWER \cite{MatP11}. In each of these systems, we apply Algorithm 1 to generate a set of PMU locations greedily, with the number of PMUs $M$ increasing from one until all attacks become observable. Moreover, from Algorithm 1, for all $M$, the minimum sparsity of unobservable attacks as well as the maximum potential impact among the sparsest unobservable attacks are found (cf. Step 1 in Algorithm 1). We assume that all buses are alterable in the test cases.

In general, for a given set of PMUs, one can also search for the maximum potential impact among all $s$-sparse unobservable attacks for any given sparsity $s$, (as opposed to evaluate that among the sparsest attacks only as in Algorithm 1). However, this problem is NP-hard in $s$. %While we employ an algorithm to list all vertex cuts of $\mc{G}^{N[\mc{M}]}$ with increasing sizes starting from the minimum vertex cut \cite{cutenum}, to proceed all the way to listing all the vertex cuts of $\mc{G}^{N[\mc{M}]}$ (and hence all the unobservable attacks) is NP-hard.
%Thus, we have selectively evaluated the maximum potential impacts of certain non-sparsest unobservable attacks. 
In light of this, we selectively focused on some level of sparsity of unobservable attacks that is \emph{not} minimally sparse, and evaluated their maximum potential impacts. %in the following.

\begin{figure}[tb]
  \centering
  \includegraphics[scale = 0.45]{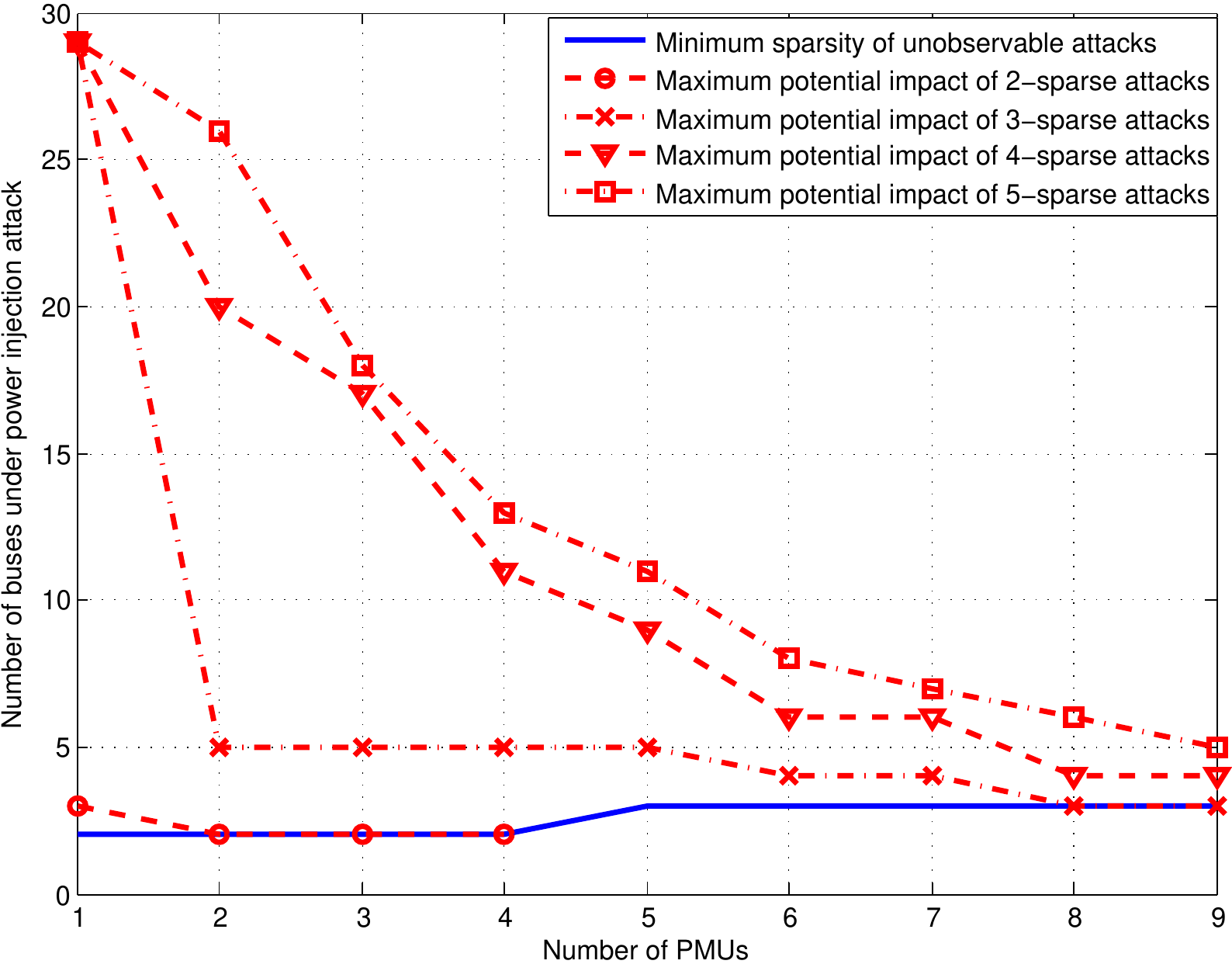}
  \caption{Minimum sparsity of unobservable attacks and maximum potential impacts of 2, 3, 4, 5-sparse attacks as functions of $M$, IEEE 30-bus system.}
  \label{eval30}
\end{figure}

\begin{figure}[tb]
  \centering
  \includegraphics[scale = 0.45]{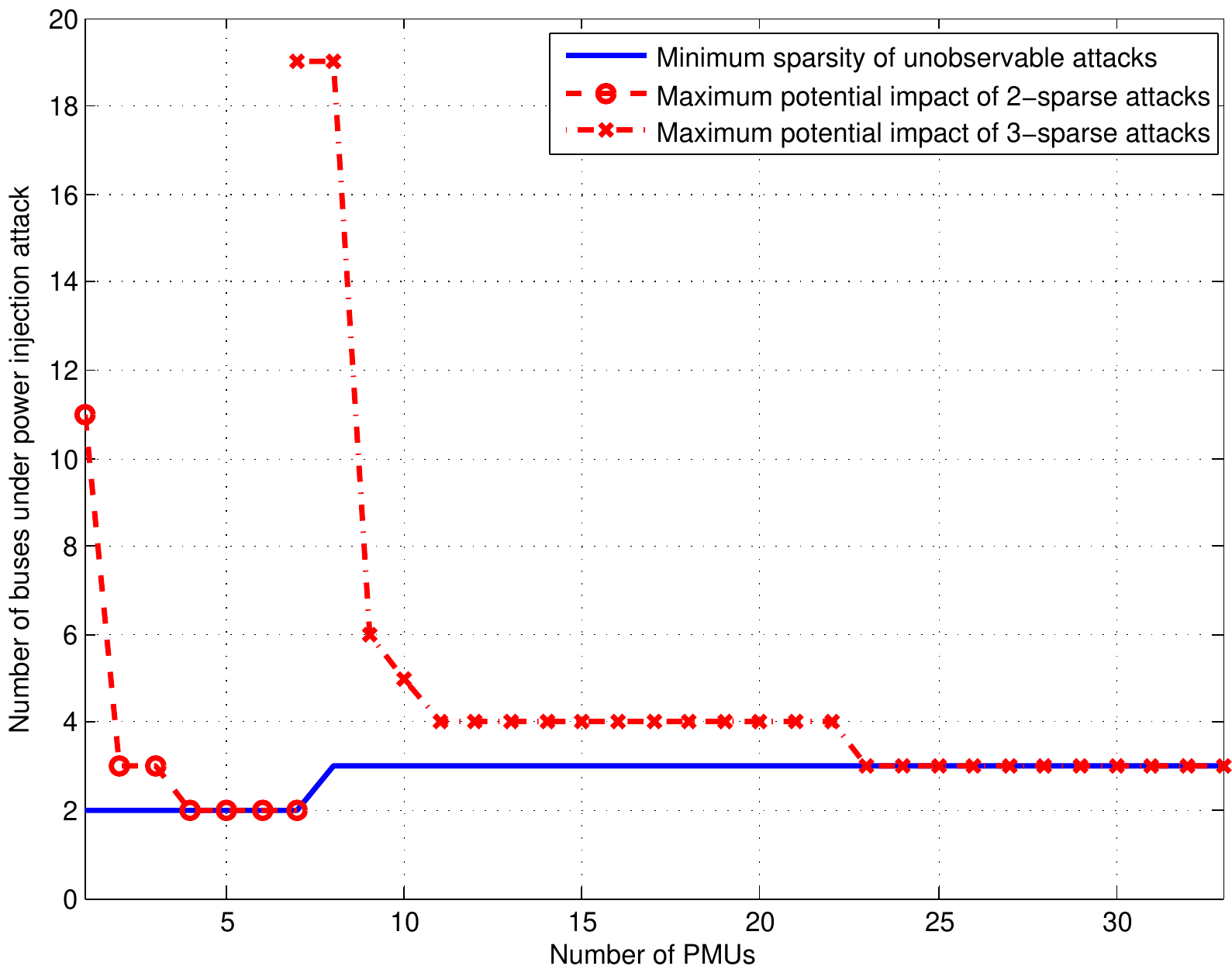}
  \caption{Minimum sparsity of unobservable attacks and the maximum potential impacts of the sparsest attacks as functions of $M$, IEEE 118-bus system.}
  \label{eval118}
\end{figure}

\begin{figure}[tb]
  \centering
  \includegraphics[scale = 0.45]{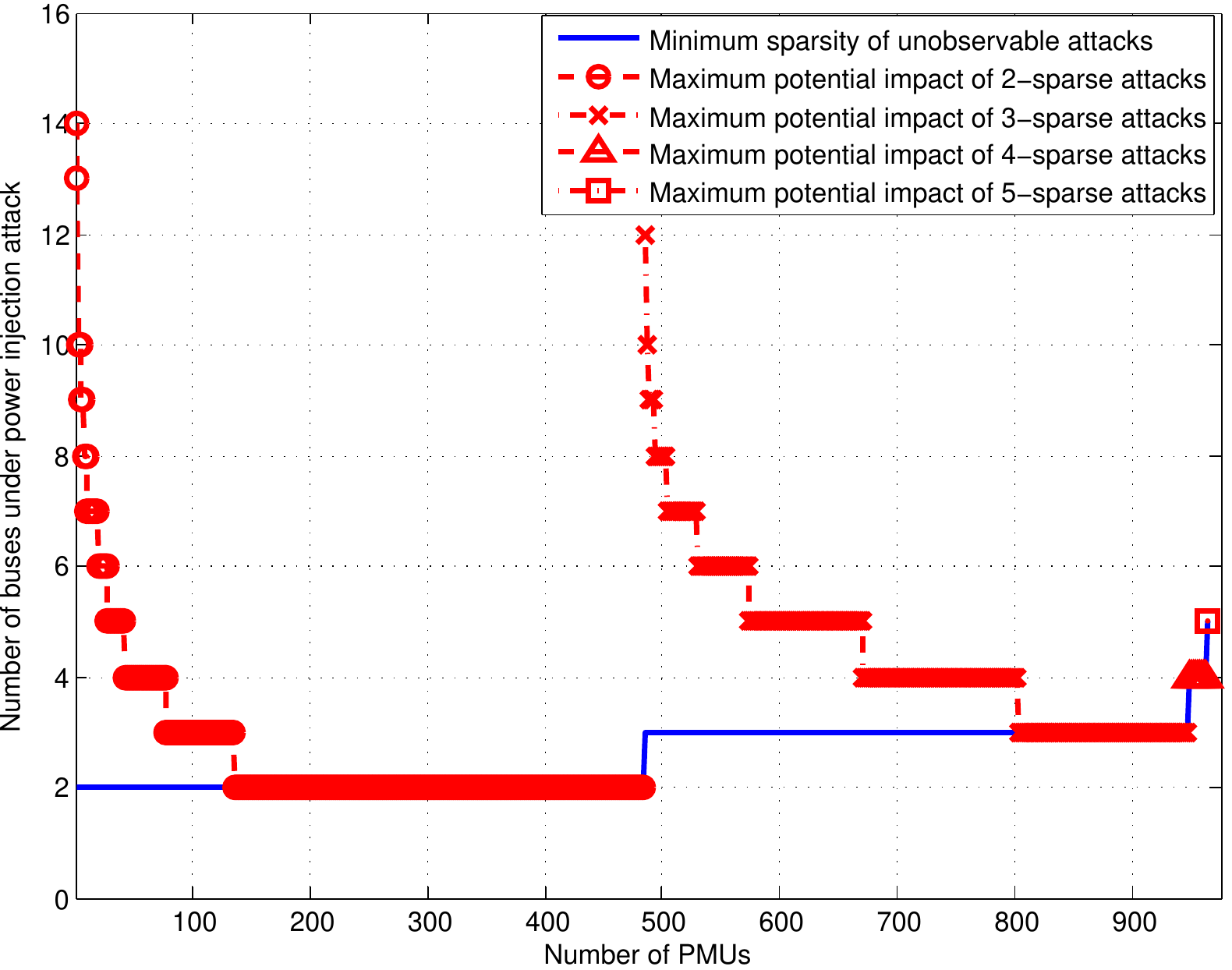}
  \caption{Minimum sparsity of unobservable attacks and the maximum potential impacts of the sparsest attacks as functions of $M$, Polish 3012-bus system.}
  \label{eval3012}
\end{figure}

Specifically, the minimum sparsity of unobservable attacks and the maximum potential impact among these sparsest attacks both as functions of the number of PMUs $M$ are plotted for the IEEE 30 and 118-bus power systems and the Polish 3012-bus system, in Figures \ref{eval30}, \ref{eval118} and \ref{eval3012} respectively. In addition,
\begin{itemize}
\item For the IEEE 30-bus system, the maximum potential impact among all 2-sparse, 3-sparse, 4-sparse and 5-sparse unobservable attacks for the entire range of $M$ are plotted. (Note that the minimum sparsity of unobservable attacks does not exceed 3 for all $M$).
%\item For the IEEE 57-bus system, the maximum potential impact among all 3-sparse attacks when $M=1$ is plotted. (Note that for $M=1$ the minimum sparsity of unobservable attacks is 2).
\item For the IEEE 118-bus system, the maximum potential impact among all 3-sparse attacks when $M\ge7$ is plotted. (Note that for $M=7$ the minimum sparsity of unobservable attacks is 2).
\end{itemize}

We make the following observations which appear in all seven of the evaluated systems:
\begin{itemize}
\item In all seven systems, all the attacks become observable with \emph{less than a third} of the buses installed with PMUs (assuming the second PMU model). The average percentage of the number of PMUs needed to have full network observability equals $31.1\%$. While this number resembles a well-known estimate of such percentage to be one third \cite{obs93}, it also demonstrates the efficacy of Algorithm 1 in PMU placement.
\item The topologies of the tested power systems tend to allow sparse power injection attacks. In other words, the vertex connectivity of these power networks is often small. Furthermore, there are often many unobservable attacks with the same minimum sparsity: %\emph{many ties} when finding unobservable attacks with the minimum sparsity: 
this is why even after adding a lot more PMUs into the network, with each addition eliminating the previous sparsest attack, the minimum sparsity of an unobservable attack 
can still remain the same.
\item While there are many %ties of 
unobservable attacks with the same sparsity, the potential impacts among them can vary significantly. Moreover, as more PMUs are added, the maximum potential impact among all the sparsest unobservable attacks drops quickly until it reaches the minimum sparsity. Similar behavior is demonstrated for all the $s$-sparse unobservable attacks ($s=2,3,4,5$) for the IEEE 30-bus system as shown in Figure \ref{eval30}. %The maximum While raising the minimum sparsity of unobservable attacks may need adding a lot of PMUs
\end{itemize}

\section{Conclusion} \label{secconc}
We have studied physical attacks that alter power generation and loads in power networks while remaining unobservable under the surveillance of system operators using PMUs. Given a set of PMUs, we have first shown that the existence of an unobservable attack that is restricted to any given subset of the buses can be determined with probability one by computing the structural rank of a submatrix of the network Laplacian $\bm{B}$. Next, we have provided an explicit solution to the open problem of finding the sparsest unobservable attacks: the minimum sparsity among all unobservable attacks equals $\kappa(\mc{G^M})+1$ with probability one. %In deriving this minimum sparsity, a lower bound on it based on the vertex connectivity of the network was shown to hold with probability one. We have then provided a constructive upper bound that successfully closes the gap to the lower bound. 
The constructive solution allows us to find all the sparsest unobservable attacks in polynomial time. %by finding the minimum vulnerable vertex cuts of an augmented graph $\mc{G^M}$. 
As a result, $\kappa(\mc{G^M})+1$ is a fundamental limit of this minimum sparsity that is not only explicitly attainable, but also unbeatable by all possible unobservable attacks. We have then introduced a notion of potential impacts of unobservable attacks. %We have further shown that the graph-theoretic interpretation of the unobservable attacks allows a natural characterization of their potential %damage
%impacts. 
For the system operator to raise the minimum sparsity while simultaneously mitigating the maximum potential impact of all unobservable attacks, we have devised an efficient algorithm of greedily placing the PMUs.
%For all possible numbers of PMUs, under their surveillance at
With optimized PMU deployment, %given by the developed algorithm, %numerical evaluation of
we have evaluated the sparsest unobservable attacks and their potential impacts %provided
in the IEEE 30, 57, 118, 300-bus systems and the Polish 2383, 2737, 3012-bus systems. Finally, while this work has studied a static system model and power injection attacks, extension to dynamic systems, measurements and power injection attacks remains an interesting future direction, for which we expect that similar insights will apply. %on fundamental limits

{\small
\appendices

\section{Proof of Lemma \ref{lemL}} \label{secprflemL}
\begin{proof}[Proof of Lemma \ref{lemL}]
First, we denote the Laplacian of the \emph{induced subgraph} $\mc{G}[\mc{I}]$ by $\bm{L}_\mc{I}$. Denote the number of connected components of the induced subgraph $\mc{G}[\mc{I}]$ by $c$. By properly re-indexing the nodes, we have
\begin{align}
\bm{B}_\mc{II} = \bm{L}_\mc{I} + \bm{D}_\mc{I},
\end{align}
where $\bm{L}_\mc{I}$ is a block-diagonal matrix whose each block $\bm{L}_\mc{I}^j~(1\le j\le c)$ is \emph{positive semidefinite} and corresponds to one connected component of $\mc{G}[\mc{I}]$,
\begin{align}
\bm{L}_\mc{I} =
\begin{bmatrix}
\bm{L}_\mc{I}^1 &&& \\
&\bm{L}_\mc{I}^2 && \\
&&\cdots&\\
&&&\bm{L}_\mc{I}^c
\end{bmatrix},
\end{align}
and $\bm{D}_\mc{I}$ is diagonal, which we write in a block diagonal form whose each block $\bm{D}_\mc{I}^j~(1\le j\le c)$ is itself a diagonal matrix with \emph{non-negative} entries,
\begin{align}
\bm{D}_\mc{I} =
\begin{bmatrix}
\bm{D}_\mc{I}^1 &&& \\
&\bm{D}_\mc{I}^2 && \\
&&\cdots&\\
&&&\bm{D}_\mc{I}^c
\end{bmatrix}.
\end{align}
%where $\bm{D}_\mc{I}$ is a diagonal matrix that has non-negative diagonal entries with at least one of them strictly positive.

Since the original graph $\mc{G}$ is connected, each connected component of the induced subgraph $\mc{G}[\mc{I}]$ must be connected to at least one node in $\mc{N}\backslash\mc{I}$. This implies the following fact:

\begin{fct}
Each diagonal submatrix $\bm{D}_\mc{I}^j~(1\le j\le c)$ has at least one strictly positive diagonal entry.
\end{fct}

Now, for any non-zero vector $\bm{x}_\mc{I}\in\mbb{R}^I$, we write it as a concatenation of $c$ sub-vectors:
\begin{align}
\bm{x}_\mc{I} = [[\bm{x}_\mc{I}^1]^T~ [\bm{x}_\mc{I}^2]^T \ldots [\bm{x}_\mc{I}^c]^T]^T,
\end{align}
where the length of each sub-vector $\bm{x}_\mc{I}^j~(1\le j\le c)$ follows the size of the sub-matrix $\bm{L}_\mc{I}^j$.

As $\bm{L}_\mc{I}$ is positive semidefinite, $\bm{x}_\mc{I}^T\bm{L}_\mc{I}\bm{x}_\mc{I} \ge 0$:
\begin{enumerate}
\item If $\bm{x}_\mc{I}^T\bm{L}_\mc{I}\bm{x}_\mc{I} > 0$, then immediately $\bm{x}_\mc{I}^T\bm{B}_\mc{II}\bm{x}_\mc{I} > 0$.
\item If $\bm{x}_\mc{I}^T\bm{L}_\mc{I}\bm{x}_\mc{I} = 0$, then $\bm{L}_\mc{I}\bm{x}_\mc{I} = 0$, which implies
\begin{align}
\bm{L}_\mc{I}^j\bm{x}_\mc{I}^j = 0, ~\forall j=1,2,\ldots,c.
\end{align}
Namely, $\bm{x}_\mc{I}^j$ is in the null space of $\bm{L}_\mc{I}^j$. Note that as $\bm{L}_\mc{I}^j$ corresponds to a single connected component of $\mc{G}[\mc{I}]$, the dimension of the null space of $\bm{L}_\mc{I}^j$ is one, and is spanned by the all one vector $\bm{1} = [1, 1, \ldots, 1]^T$ with the appropriate length. Thus, $\bm{x}_\mc{I}^j$ must be in the form of $\alpha_j\cdot\bm{1}$, for some $\alpha_j>0$. From Fact 1, $\bm{D}_\mc{I}^j$ has non-negative diagonal entries with at least one of them strictly positive, and we have $[\bm{x}_\mc{I}^j]^T \bm{D}_\mc{I}^j\bm{x}_\mc{I}^j > 0$, and hence $\bm{x}_\mc{I}^T(\bm{B}_\mc{II})\bm{x}_\mc{I} = \bm{x}_\mc{I}^T(\bm{L}_\mc{I}+\bm{D}_\mc{I})\bm{x}_\mc{I} > 0$.
\end{enumerate}

Therefore, $\bm{B}_\mc{II}$ is positive definite, and hence of full rank.
\end{proof}

\section{Proof of Theorem \ref{feasthm}} \label{secprffeasthm}
First, for a matrix $\bm{H}\in\mbb{R}^{N_1\times N_2}$ with a full structural rank, we define an equivalent term, ``a non-zero permuted diagonal'', for a set of $\min(N_1, N_2)$ independent entries (cf. Definition \ref{indent}). This term is based on the following intuition: For example, for $\bm{H}\in\mbb{R}^{N\times N}$, a non-zero permuted diagonal (i.e., a set of $N$ independent entries) corresponds to a permutation function $\pi(i), i=1,2,\ldots,N$, such that $\bm{H}_{i,\pi(i)} \ne 0, \forall i=1,2,\ldots,N$.

\begin{proof}[Proof of Theorem \ref{feasthm}]
It is sufficient to prove for the case of $N' = N'' \le N-1$.
We use induction as follows.

i) Clearly, any non-zero $1\times1$ submatrix of $\bm{B}$ is of full rank.

ii) Assume that all $t\times t~ (t\le N-2)$ submatrices of $\bm{B}$ with a non-zero permuted diagonal are of full rank with probability one.

For a $(t+1)\times (t+1)$ submatrix of $\bm{B}$ with a non-zero permuted diagonal, we denote it by $\bm{B}'$. We denote the set of row indices of $\bm{B}$ that are selected in forming $\bm{B}'$ by $\mc{R} = \{r(1),r(2),\ldots,r(t+1)\}$, and similarly the set of selected column indices by $\mc{C} = \{c(1),c(2),\ldots,c(t+1)\}$: $\bm{B}'_{i,j} = \bm{B}_{r(i),c(j)}, \forall 1\le i,j\le t+1.$
Clearly, if %$\bm{B}'$ is formed by selecting the \emph{same} set of row indices and column indices of $\bm{B}$,
$\mc{R} = \mc{C}$, $\bm{B}'$ is of full rank from Lemma \ref{lemL}.

\begin{figure}[tb!]
  \centering
  \subfigure[Partition of the matrix $\bm{B}'$.]{ %when the set of the row indices are not the same as the set of the column indices
  \includegraphics[scale = 0.9]{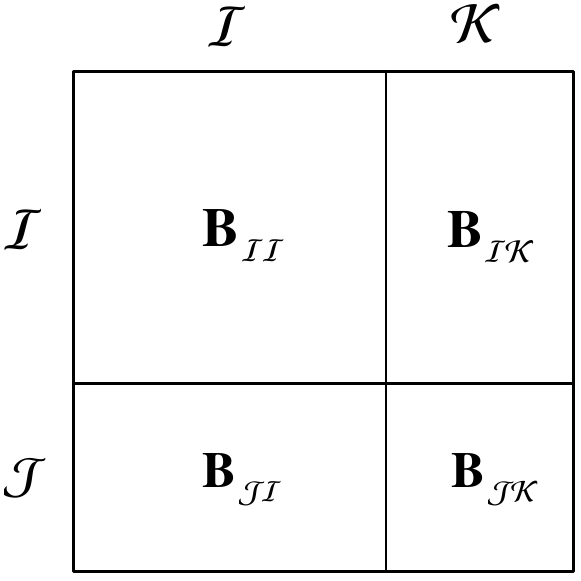}
  \label{IJK11}
  }
  \subfigure[Case 1, $\bm{B}'_{t+1,t+1}\in\bm{B}_{\mc{JK}}$ is on a non-zero permuted diagonal of $\bm{B}'$.]{
  \includegraphics[scale = 0.9]{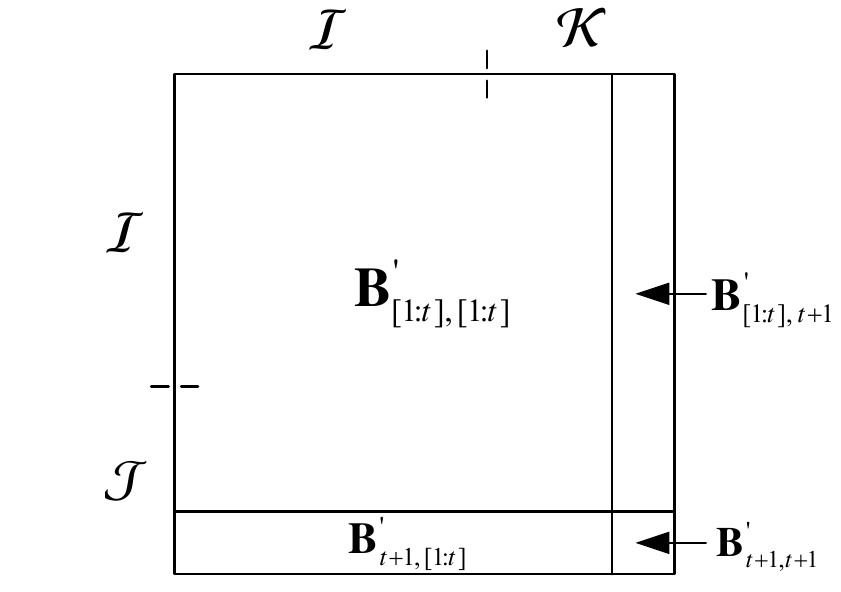}
  \label{IJK12}
  }
  \caption{Proof of Theorem \ref{feasthm}.}
  \label{mainlemcase1}
\end{figure}

Now, consider the case that %$\bm{B}'$ is formed by selecting the rows of $\bm{B}$ with row indices
$\mc{R} = \mc{I}\cup\mc{J}, ~ \mc{C} = \mc{I}\cup\mc{K}$, %and the columns of $\bm{B}$ with column indices $\mc{I}\cup\mc{K}$,
where $\mc{I}\cap\mc{J} = \mc{I}\cap\mc{K} = \emptyset, ~\mc{J}\cap\mc{K} = \emptyset,~ \mc{J},\mc{K}\ne\emptyset$. % , ~|\mc{I}|+|\mc{J}| = |\mc{I}|+|\mc{K}| = N'
In other words, $\mc{I}$ denotes the common indices that appear in both the row indices $\mc{R}$ and the column indices $\mc{C}$, $\mc{J}$ denotes the indices that appear in $\mc{R}$ but not in $\mc{C}$, and $\mc{K}$ denotes the indices that appear in $\mc{C}$ but not in $\mc{R}$.
WLOG, $\bm{B}'$ has the form as in Figure \ref{IJK11}, in which the common row and column indices $\mc{I}$ are located in the upper left part of $\bm{B}'$, and $\bm{B}'$ consists of four blocks $\bm{B}_{\mc{II}}, \bm{B}_{\mc{JI}}, \bm{B}_{\mc{IK}}, \bm{B}_{\mc{JK}}$.

Since $\bm{B}'$ has a non-zero permuted diagonal, there exists a permutation function $\pi(i), i =1,2,\ldots,t+1$, such that $\bm{B}'_{i,\pi(i)}>0, \forall i=1,\ldots,t+1$. In other words, the mapping $r(i)\rightarrow c(\pi(i)), i =1,2,\ldots,t+1$ forms a bijection between $\mc{I}\cup\mc{J}$ and $\mc{I}\cup\mc{K}$, such that $\bm{B}_{r(i),c(\pi(i))}>0, \forall i = 1,2,\ldots,t+1.$
We now consider the following two cases:

\emph{Case 1:} $\exists r(i)\in\mc{J}, ~s.t.~ c(\pi(i))\in\mc{K}.$ In other words, one of the entries in $\bm{B}_{\mc{JK}}$ is on a non-zero permuted diagonal of $\bm{B}'$.

WLOG, assume that the entry $\bm{B}'_{t+1,t+1}$ in $\bm{B}_{\mc{JK}}$ is on a non-zero permuted diagonal of $\bm{B}'$, i.e., $\pi(t+1) = t+1$. As in Figure \ref{IJK12}, we partition $\bm{B}'$ into $\bm{B}'_{[1:t], [1:t]}, \bm{B}'_{t+1, [1:t]}, \bm{B}'_{[1:t], t+1}$ and $\bm{B}'_{t+1,t+1}$. Because $\bm{B}'_{t+1,t+1}$ is on a non-zero permuted diagonal of $\bm{B}'$, the $t\times t$ submatrix $\bm{B}'_{[1:t], [1:t]}$ has a non-zero permuted diagonal.

From the induction assumption, $\bm{B}'_{[1:t], [1:t]}$ is of full rank with probability one. When $\bm{B}'_{[1:t], [1:t]}$ is of full rank, let
\begin{align}
\bm{\alpha} = {\bm{B}'}^{-1}_{[1:t], [1:t]}\bm{B}'_{[1:t], t+1}.
\end{align}
Thus, $\bm{B}'_{[1:t], t+1} = \bm{B}'_{[1:t], [1:t]}\bm{\alpha}$. Then, $\bm{B}'$ is rank-deficient if and only if
\begin{align}
\bm{B}'_{t+1,t+1} = \bm{B}'_{t+1, [1:t]} \bm{\alpha} \label{rankdef1}
\end{align}

Note that, except for $\bm{B}_{r(t+1),c(t+1)} = \bm{B}'_{t+1,t+1}$ itself, there are only three other entries in the Laplacian $\bm{B}$ that are correlated with $\bm{B}'_{t+1,t+1}$:
\begin{align}
\bm{B}_{c(t+1),r(t+1)} = \bm{B}_{r(t+1),c(t+1)} = \bm{B}'_{t+1,t+1}, \\
\bm{B}_{r(t+1),r(t+1)} = -\sum_{j\ne r(t+1)} \bm{B}_{r(t+1), j}, \\
\bm{B}_{c(t+1),c(t+1)} = -\sum_{i\ne c(t+1)} \bm{B}_{i, c(t+1)}.
\end{align}
However, as $r(t+1)\in\mc{J}\Rightarrow r(t+1)\notin\mc{I}\cup\mc{K},~ c(t+1)\in\mc{K}\Rightarrow c(t+1)\notin\mc{I}\cup\mc{J}$, \emph{none} of the above three entries is selected into the submatrix $\bm{B}'$. Therefore, \emph{$\bm{B}'_{t+1,t+1}$ is independent to all other entries in $\bm{B}'$}, and is hence independent to $\bm{B}'_{t+1, [1:t]} \bm{\alpha}$. Because $\bm{B}'_{t+1,t+1}$ is drawn from a continuous distribution, the probability that \eqref{rankdef1} is satisfied is zero. As a result, $\bm{B}'$ is of full rank with probability one.
\vspace{10pt}

\emph{Case 2:} $\forall r(i)\in\mc{J}, ~ c(\pi(i))\notin\mc{K}.$ Thus, $\exists r(i)\in\mc{I},~s.t.~c(\pi(i))\in\mc{K}$. In other words, one of the entries in $\bm{B}_{\mc{IK}}$ is on a non-zero permuted diagonal of $\bm{B}'$.

\begin{figure}[tb!]
  \centering
  \subfigure[Case 2, $\bm{B}'_{1,t+1}\in\bm{B}_{\mc{IK}}$ is on a non-zero permuted diagonal of $\bm{B}'$.]{
  \includegraphics[scale = 0.9]{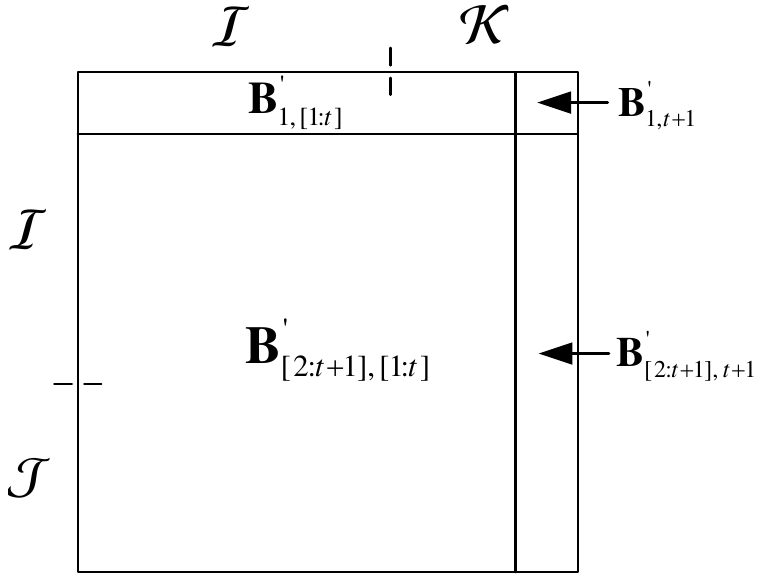}
  \label{IJK21}
  }
  \subfigure[Case 2, $\bm{B}'_{i, i}$ depends on $\bm{B}'_{i, 1},\bm{B}'_{i, t+1}$ only via the \emph{sum} $\bm{B}'_{i, 1} + \bm{B}'_{i, t+1}$.]{
  \includegraphics[scale = 0.9]{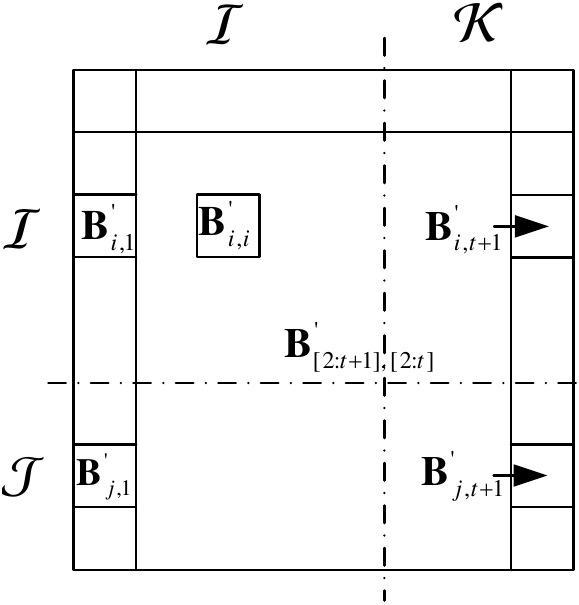}
  \label{IJK22}
  }
  \caption{Proof of Theorem \ref{feasthm}.}
  \label{mainlemcase2}
\end{figure}

WLOG, assume that the entry $\bm{B}'_{1,t+1}$ in $\bm{B}_{\mc{IK}}$ is on a non-zero permuted diagonal of $\bm{B}'$, i.e., $\pi(1) = t+1$. As in Figure \ref{IJK21}, we partition $\bm{B}'$ into $\bm{B}'_{1,[1:t]}, \bm{B}'_{[2:t+1], [1:t]}, \bm{B}'_{1, t+1}$ and $\bm{B}'_{[2:t+1],t+1}$. Because $\bm{B}'_{1,t+1}$ is on a non-zero permuted diagonal of $\bm{B}'$, the $t\times t$ submatrix $\bm{B}'_{[2:t+1], [1:t]}$ has a non-zero permuted diagonal.

From the induction assumption, $\bm{B}'_{[2:t+1], [1:t]}$ is of full rank with probability one. When $\bm{B}'_{[2:t+1], [1:t]}$ is of full rank, let
\begin{align}
\bm{\alpha} = {\bm{B}'}^{-1}_{[2:t+1], [1:t]}\bm{B}'_{[2:t+1], t+1}. \label{alphadef}
\end{align}
Thus, $\bm{B}'_{[2:t+1], t+1} = \bm{B}'_{[2:t+1], [1:t]}\bm{\alpha}$. Then, $\bm{B}'$ is rank-deficient if and only if
\begin{align}
\bm{B}'_{1,t+1} = \bm{B}'_{1, [1:t]} \bm{\alpha} \label{rankdef2}
\end{align}

Note that %$1\in\mc{I}$, and
$r(1) = c(1)\in\mc{I}$. We have
\begin{align}
\bm{B}'_{1,1} &= \bm{B}_{r(1),c(1)} = \bm{B}_{r(1),r(1)} = -\sum_{j\ne r(1)}\bm{B}_{r(1),j} \nn\\
&= - \bm{B}'_{1,t+1} - C_1, \label{bp11}
\end{align}
where $\bm{B}'_{1,t+1} = \bm{B}_{r(1),c(t+1)}$, and $C_1 = \sum_{j\ne r(1), j\ne c(t+1)}\bm{B}_{r(1),j}$ is \emph{independent} to $\bm{B}'_{1,t+1}$. Substitute \eqref{bp11} for $\bm{B}'_{1,1}$ in \eqref{rankdef2}, we have
\begin{align}
&\bm{B}'_{1,t+1} = \alpha_1\bm{B}'_{1,1} + \sum_{j=2}^{t}\alpha_j\bm{B}'_{1,j} \nn\\
&~~~~~~~~= - \alpha_1\bm{B}'_{1,t+1} + \left(- \alpha_1C_1 + \sum_{j=2}^{t}\alpha_j\bm{B}'_{1,j}\right). \nn\\
\Leftrightarrow~& (1+\alpha_1)\bm{B}'_{1,t+1} = - \alpha_1C_1 + \sum_{j=2}^{t}\alpha_j\bm{B}'_{1,j}. \label{rankdef3}
\end{align}
Note that $\bm{B}'_{1,t+1}$ is \emph{independent to $\alpha_1$, and independent to the right hand side of \eqref{rankdef3}}. Because $\bm{B}'_{1,t+1}$ is drawn from a continuous distribution, if $\alpha_1\ne -1$, the probability (conditioned on $\alpha_1\ne -1$) that \eqref{rankdef3} is satisfied is zero.
\vspace{10pt}

Next, we prove that the probability of $\alpha_1 = -1$ is zero. From \eqref{alphadef}, if $\alpha_1 = -1$,
\begin{align}
&\bm{B}'_{[2:t+1], t+1} = \sum_{j=1}^t\alpha_j\bm{B}'_{[2:t+1], j} \nn\\
\Leftrightarrow~ & \bm{B}'_{[2:t+1], 1} + \bm{B}'_{[2:t+1], t+1} = \sum_{j=2}^t\alpha_j\bm{B}'_{[2:t+1], j}.
\end{align}
Thus, $\alpha_1 = -1$ implies that $\bm{B}'_{[2:t+1], 1} + \bm{B}'_{[2:t+1], t+1}$ is in the range space of $\bm{B}'_{[2:t+1], [2:t]}$. Note that, all the entries in the two vectors $\bm{B}'_{[2:t+1], 1}$ and $\bm{B}'_{[2:t+1], t+1}$ are mutually independent non-diagonal entries of $\bm{B}$. %Thus, their sum $\bm{b}$ has mutually independent entries.

%For any given distributions of the entries in $\bm{B}$, let the
Now, consider that we make the following \emph{change of distributions} of certain entries in $\bm{B}'$ (and also $\bm{B}$ correspondingly):
\begin{enumerate}
\item $\forall i = 2,\ldots,t+1$, let $\bm{B}'_{i, 1} (= \bm{B}_{r(i), c(1)})$ be drawn from the distribution of the \emph{sum} $\bm{B}'_{i, 1} + \bm{B}'_{i, t+1}$.
\item Let every entry of $\bm{B}'_{[2:t+1], t+1}$ to be a deterministic \emph{zero}, i.e., $\forall i = 2,\ldots,t+1,~ \bm{B}'_{i, t+1} (= \bm{B}_{r(i), c(t+1)}) = 0.$
\end{enumerate}
Observe that,
\begin{itemize}
\item As in Figure \ref{IJK22}, $\forall i, ~s.t.~ r(i)\in\mc{I}$, the \emph{only} entry in $\bm{B}'_{[2:t+1], [2:t]}$ that is correlated with $\bm{B}'_{i, 1}$ and $\bm{B}'_{i, t+1}$ is
    \begin{align}
    \bm{B}'_{i, i} = -(\bm{B}'_{i, 1} + \bm{B}'_{i, t+1}) - C_2,
    \end{align}
    where $C_2 = \sum_{j\ne c(1),j\ne c(t+1)}\bm{B}_{r(i),j}$. Note that $\bm{B}'_{i, i}$ depends on $\bm{B}'_{i, 1},\bm{B}'_{i, t+1}$ only via the \emph{sum} $\bm{B}'_{i, 1} + \bm{B}'_{i, t+1}$.
\item $\forall i, ~s.t.~ r(i)\in\mc{J}$, $\bm{B}'_{i, 1}$ and $\bm{B}'_{i, t+1}$ are independent to all the entries in $\bm{B}'_{[2:t+1], [2:t]}$.
\end{itemize}
This implies the following fact:

\begin{fct} The joint distribution of $\bm{B}'_{[2:t+1], 1}$ and $\bm{B}'_{[2:t+1], [2:t]}$ after the change of distributions is equal to the joint distribution of $\bm{B}'_{[2:t+1], 1} + \bm{B}'_{[2:t+1], t+1}$ and $\bm{B}'_{[2:t+1], [2:t]}$ before the change.
\end{fct}

We now note that, after the change of distributions, the $t\times t$ matrix $\bm{B}'_{[2:t+1], [1:t]} = [\bm{B}'_{[2:t+1], 1} ~\bm{B}'_{[2:t+1], [2:t]}]$ still satisfies the induction assumption, and is hence of \emph{full rank with probability one}. Thus, before the change of distributions, $\bm{B}'_{[2:t+1], 1} + \bm{B}'_{[2:t+1], t+1}$ falls in the range space of $\bm{B}'_{[2:t+1], [2:t]}$ \emph{with probability zero}. Therefore, $\alpha_1 = -1$ with probability zero, hence the probability that \eqref{rankdef2} is satisfied is zero. As a result, $\bm{B}'$ is of full rank with probability one.
\end{proof}

\section{Proof of Lemma \ref{cond1equ}} \label{secprfcond1}
%Before proving that \eqref{cond1} is equivalent to the non-zero permuted diagonal property of $\bm{A}$,
We first prove the following lemma in preparation for proving Lemma \ref{cond1equ}:
\begin{lem}\label{lemcond1}
For a matrix $\bm{H}\in\mbb{R}^{N\times N}$, if the following conditions are satisfied,
\begin{align}
&\bm{H}_{1,1} \ne0, \label{cond11}\\
&\forall i=2,\ldots,N,\bm{H}_{i,i-1} \ne0, \label{cond12}\\
&\forall i=2,\ldots,N, \text{ the sub-column } \bm{H}_{[1,i],i} \text{ has at least one} \nn\\
&~~~~~~~~~~~~~~~~~~\text{non-zero entry,} \label{cond13}
\end{align}

then $\bm{H}$ satisfies Property \ref{cond1def}.
\end{lem}
A depiction of a matrix satisfying \eqref{cond11}, \eqref{cond12} and \eqref{cond13} is given in Figure \ref{cond1fig}, in which the entries with an ``x'' are known to be non-zero, and the shaded sub-columns each has at least one non-zero entry.

\begin{figure}[htb!]
  \centering
  \includegraphics[scale=0.8]{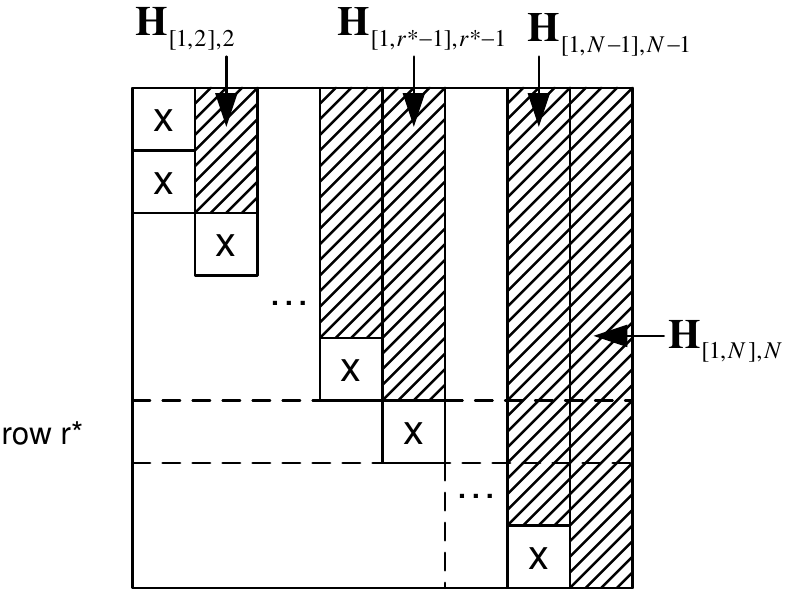}
  \caption{A matrix that satisfies \eqref{cond11}, \eqref{cond12} and \eqref{cond13} as in Lemma \ref{lemcond1}.}
  \label{cond1fig}
\end{figure}

\begin{proof}
We use induction as follows.

i) The lemma is true for $N = 1$.

ii) Assume that the lemma is true for all $N = 1,\ldots,t$. For $N = t+1$:% (cf. Figure \ref{}):

First, because the upper left $(N-1)\times (N-1)$ submatrix of $\bm{H}$ satisfies the induction assumption, each of $\bm{H}$'s first left $N-1$ columns must contain at least one non-zero entry. From \eqref{cond13}, the last column of $\bm{H}$ has at least one non-zero entry. Thus, the case of $n=N$ in Property \ref{cond1def} holds for $\bm{H}$.

Next, $\forall 1\le n\le N-1$, for any $n\times N$ submatrix of $\bm{H}$, denote it by $\bm{H}'$, and its corresponding row indices in $\bm{H}$ by $r(1)<r(2)<\ldots< r(n)$.
\begin{itemize}
\item If the last $n$ rows of $\bm{H}$ are selected to form $\bm{H}'$, (i.e. $r(i) = i+N-n,~ i=1,\ldots,n$), from \eqref{cond12}, the columns $N-n,\ldots,N-1$ each has one non-zero entry, namely, $\bm{H}_{N-n+1,N-n},\ldots,\bm{H}_{N,N-1}$.
\item Otherwise, there exists a row $r^*$, %an $r^{th}$ row,
$r^*\ge N-n+1$, which is not selected in $\bm{H}'$ (cf. Figure \ref{cond1fig}). In this case, the row indices of $\bm{H}'$ can be partitioned into two subsets: $\exists i(\in \{1,2,\ldots,n\}), r(1)<\ldots< r(i)\le r^*-1$ and $r^*+1\le r(i+1)<\ldots\le r(n)$. On the one hand, note that the upper left $(r^*-1)\times (r^*-1)$ submatrix of $\bm{H}$ satisfies the induction assumption. Thus, \emph{among the first $r^*-1$ columns} of the rows $r(1),\ldots, r(i)$, there exists $i$ columns each of which has one non-zero entry. On the other hand, from \eqref{cond12}, $\bm{H}_{r(i+1),r(i+1)-1},\ldots,\bm{H}_{r(n),r(n)-1}$ are all non-zero, and none of these non-zero entries appears in the first $r^*-1$ columns. Therefore, there exist $n$ columns of $\bm{H}'$ such that each of them has at least one non-zero entry.
\end{itemize}
\end{proof}

We now prove Lemma \ref{cond1equ}.
\begin{proof}[Proof of Lemma \ref{cond1equ}]
Clearly, Property \ref{cond1def} is implied by the non-zero diagonal property.

To prove that the non-zero diagonal property is also implied by Property \ref{cond1def}, we use induction on $\bm{N}$ as follows.

i) For $N=1$, the non-zero diagonal property is implied by Property \ref{cond1def}.

ii) Assume that $\forall N \le t, (t\ge1,)$ the non-zero diagonal property is implied by Property \ref{cond1def}.

For $N = t+1$, we use another induction on the number of rows $n$ of submatrices of $\bm{H}$ in proving a %generalized
non-zero permuted diagonal property.
\vspace{10pt}

a) For $n=1$, directly from Property \ref{cond1def}, any $n\times N$ submatrix of $\bm{H}$ (i.e., any row of $\bm{H}$) has at least one non-zero entry.

b) Assume that $\forall n \le t <N$, $\forall n\times N$ submatrix of $\bm{H}$, denoted by $\bm{H}'$, it has the following property:
\begin{align}
&\exists \pi(i)~ (i = 1,\ldots, n) \text{ that satisfies } \pi(i)\ne\pi(j), \forall i\ne j, \nn\\
&s.t.~ \bm{H}'_{i,\pi(i)} \ne 0. \label{nonsquare}
\end{align}
%We call such $\bm{H}'_{i,\pi(i)} ~(i = 1,\ldots, n)$ a \emph{generalized non-zero permutated diagonal} of $\bm{H}'$ which is in general a \emph{non-square} matrix. When $n=N$, \eqref{nonsquare} is exactly the non-zero permutated diagonal property of $\bm{H}$.

\begin{figure}[htb!]
  \centering
  \includegraphics[scale=0.7]{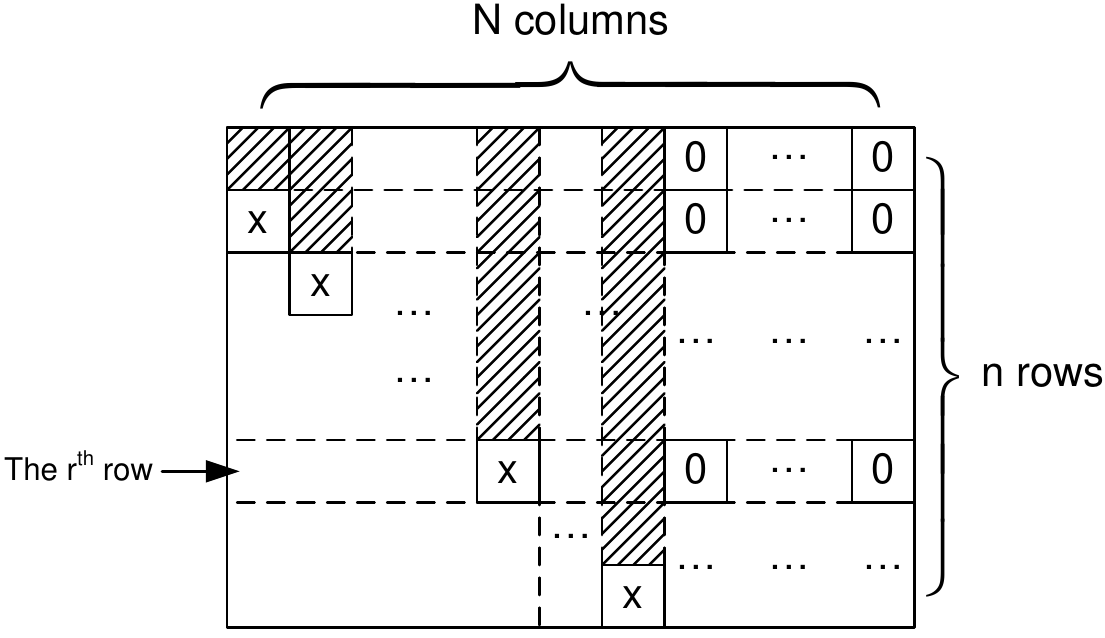}
  \caption{The matrix $\bm{H}'$ in the proof of Lemma \ref{cond1equ}.}
  \label{cond1fig2}
\end{figure}

For $n = t+1$, from the induction assumption b), there exists a %generalized
non-zero permutated diagonal for the $(n-1)\times N$ submatrix $\bm{H}'_{[2:n],[1:N]}$: WLOG, assume that it corresponds to $\forall i=2,\ldots,n, \bm{H}'_{i,i-1} \ne 0$ (cf. Figure \ref{cond1fig2}.)

Now, we use proof by contradiction, and assume that $\bm{H}'$ does \emph{not} have a %generalized
non-zero permuted diagonal. Then, the sub-row $\bm{H}'_{1,[n,N]}$ must be all zero, because otherwise any non-zero entry within $\bm{H}'_{1,[n,N]}$ will form a %generalized
non-zero permuted diagonal of $\bm{H}'$ with $\bm{H}'_{i,i-1} (i=2,\ldots,n)$. From Property \ref{cond1def}, the $1^{st}$ row of $\bm{H}'$ must have at least one non-zero entry. WLOG, assume that $\bm{H}'_{1,1}$ is non-zero. Then, the sub-row $\bm{H}'_{2,[n,N]}$ must be all-zero, because otherwise any non-zero entry within $\bm{H}'_{2,[n,N]}$ will form a %generalized
non-zero permuted diagonal of $\bm{H}'$ with $\bm{H}'_{1,1}$ and $\bm{H}'_{i,i-1} (i=3,\ldots,n)$. From Property \ref{cond1def}, the first two rows of $\bm{H}'$ must have at least two columns each of which has at least one non-zero entry. Since $\bm{H}'_{2,1} \ne 0$, there is at least one more column of $\bm{H}'_{[1,2],[1,N]}$ that has at least one non-zero entry. WLOG, assume that the sub-column $\bm{H}'_{[1,2],2}$ has at least one non-zero entry, (cf. the shaded area in Figure \ref{cond1fig2}).

Similarly, consider the $r^{th}$ row of $\bm{H}'$. Note that the submatrix $\bm{H}'_{[1:r-1], [1:r-1]}$ satisfies \eqref{cond11}, \eqref{cond12} and \eqref{cond13}, and hence satisfies Property \ref{cond1def} by Lemma \ref{lemcond1}. From the induction assumption ii), $\bm{H}'_{[1:r-1], [1:r-1]}$ has a non-zero permuted diagonal. Then, the sub-row $\bm{H}'_{r,[n,N]}$ must be all-zero, because otherwise any non-zero entry within $\bm{H}'_{r,[n,N]}$ will form a %generalized
non-zero permuted diagonal of $\bm{H}'$ with the non-zero permuted diagonal of $\bm{H}'_{[1:r-1], [1:r-1]}$ and $\bm{H}'_{i,i-1} (i=r+1,\ldots,n)$.

Therefore, the submatrix $\bm{H}'_{[1,N],[n,N]}$ must be all-zero. This implies that there are only $n-1$ (instead of $n$) columns of $\bm{H}'$ each of which has at least one non-zero entry, and hence contradicts with Property \ref{cond1def}.
\end{proof}

%\bibliographystyle{IEEEbib}
%{\bibliography{ICASSP12}}

\bibliographystyle{IEEEtran}
{\bibliography{PHYATKJ}}
}

\begin{IEEEbiography}[{\includegraphics[width=1in,height=1.25in,clip,keepaspectratio]{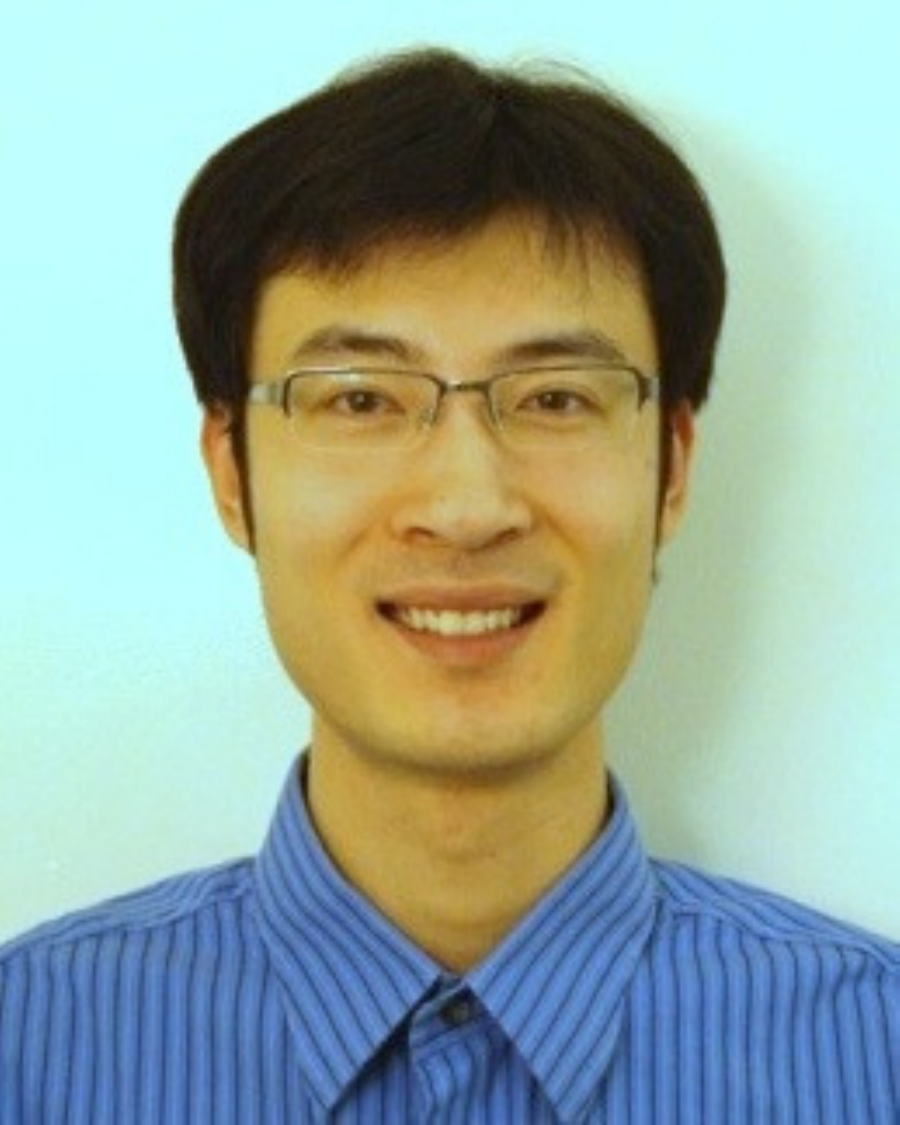}}]{YUE ZHAO}
(S'06, M'11) is an assistant professor of Electrical and Computer Engineering at Stony Brook University. He received the B.E. degree in Electronic Engineering from Tsinghua University, Beijing, China in 2006, and the M.S. and Ph.D. degrees in Electrical Engineering from the University of California, Los Angeles (UCLA), Los Angeles in 2007 and 2011, respectively. His current research interests include smart grid, renewable energy integration, optimization theory, stochastic control, and statistical signal processing. 
\end{IEEEbiography}

\begin{IEEEbiography}[{\includegraphics[width=1in,height=1.25in,clip,keepaspectratio]{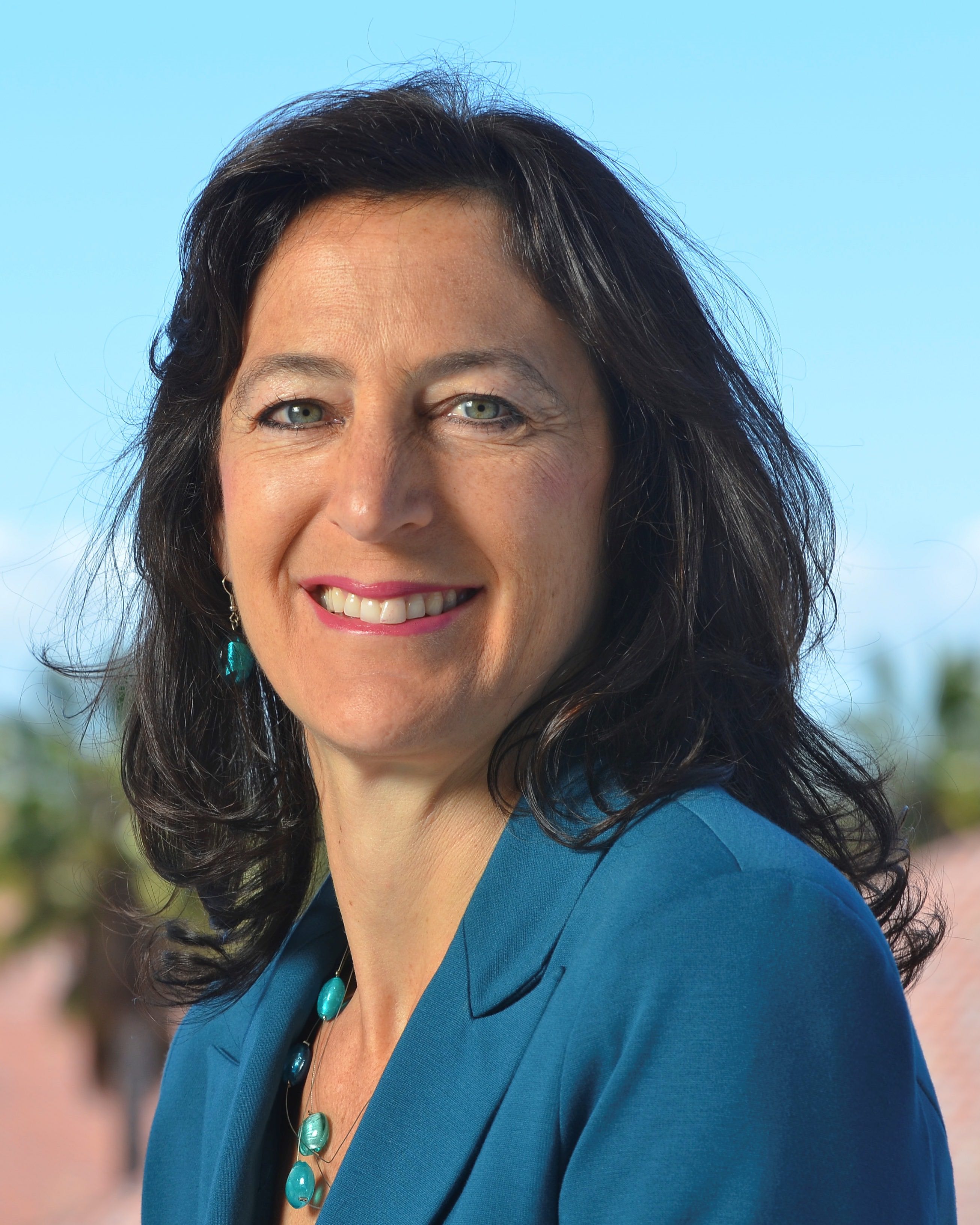}}]{Andrea Goldsmith}
(S'90, M'93, SM'99, F'05) 
is the Stephen Harris professor in the School of Engineering and a professor of Electrical Engineering at Stanford University. She was previously on the faculty of Electrical Engineering at Caltech. Dr. Goldsmith co-founded and served as CTO for two wireless companies: Wildfire.Exchange, which develops software-defined wireless network technology for cloud-based management of WiFi  systems, and  Quantenna Communications, Inc., which develops high-performance WiFi chipsets. She has previously held industry positions at Maxim Technologies, Memorylink Corporation, and AT\&T Bell Laboratories. She is a Fellow of the IEEE and of Stanford, and has received several awards for her work, including the IEEE ComSoc Edwin H. Armstrong Achievement Award as well as Technical Achievement Awards in Communications Theory and in Wireless Communications,  the National Academy of Engineering Gilbreth Lecture Award, the IEEE ComSoc and Information Theory Society Joint Paper Award, the IEEE ComSoc Best Tutorial Paper Award, the Alfred P. Sloan Fellowship, the WICE Outstanding Achievement Award, and the Silicon Valley/San Jose Business Journal's Women of Influence Award. She is author of the book ``Wireless Communications'' and co-author of the books ``MIMO Wireless Communications'' and ``Principles of Cognitive Radio,'' all published by Cambridge University Press, as well as an inventor on 28 patents. She received the B.S., M.S. and Ph.D. degrees in Electrical Engineering from U.C. Berkeley.

Dr. Goldsmith has served as editor for the IEEE Transactions on Information Theory, the Journal on Foundations and Trends in Communications and Information Theory and in Networks, the IEEE Transactions on Communications, and the IEEE Wireless Communications Magazine as well as on the Steering Committee for the IEEE Transactions on Wireless Communications. She participates actively in committees and conference organization for the IEEE Information Theory and Communications Societies and has served on the Board of Governors for both societies. She has also been a Distinguished Lecturer for both societies, served as President of the IEEE Information Theory Society in 2009, founded and chaired the Student Committee of the IEEE Information Theory Society, and chaired the Emerging Technology Committee of the IEEE Communications Society. At Stanford she received the inaugural University Postdoc Mentoring Award, served as Chair of Stanford's Faculty Senate in 2009, and currently serves on its Faculty Senate, Budget Group, and Task Force on Women and Leadership.

\end{IEEEbiography}

\begin{IEEEbiography}[{\includegraphics[width=1in,height=1.25in,clip,keepaspectratio]{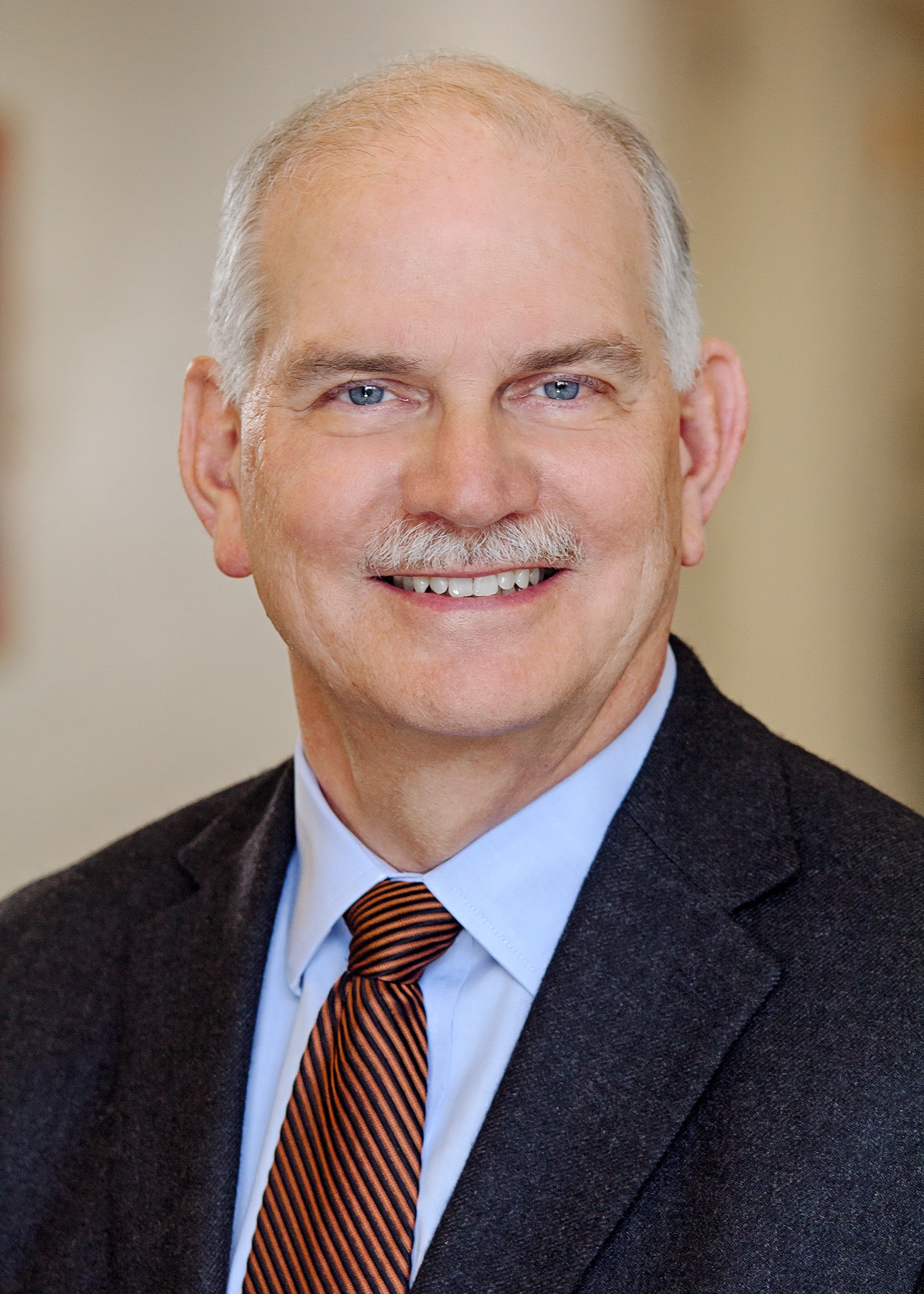}}]{H. Vincent Poor}
(S'72, M'77, SM'82, F'87) received the Ph.D. degree in electrical engineering and computer science from Princeton University in 1977.  From 1977 until 1990, he was on the faculty of the University of Illinois at Urbana-Champaign. Since 1990 he has been on the faculty at Princeton, where he is the Dean of Engineering and Applied Science, and the Michael Henry Strater University Professor of Electrical Engineering. He has also held visiting appointments at several other institutions, most recently at Imperial College and Stanford. His research interests are in the areas of information theory, stochastic analysis and statistical signal processing, and their applications in wireless networks and related fields such as smart grid. Among his publications in these areas is the recent book Mechanisms and Games for Dynamic Spectrum Allocation (Cambridge University Press, 2014).

Dr. Poor is a member of the National Academy of Engineering and the National Academy of Sciences, and is a foreign member of the Royal Society. He is also a fellow of the American Academy of Arts and Sciences and of other national and international academies. He received a Guggenheim Fellowship in 2002 and the IEEE Education Medal in 2005. Recent recognition of his work includes the 2014 URSI Booker Gold Medal, the 2015 EURASIP Athanasios Papoulis Award, the 2016 John Fritz Medal, and honorary doctorates from Aalborg University, Aalto University, HKUST, and the University of Edinburgh.

\end{IEEEbiography}

\end{document}